\documentclass{amsart}

\usepackage{amssymb,stmaryrd}
\usepackage{mathrsfs}
\usepackage{xcolor}
\usepackage[inline]{enumitem}
\usepackage{tikz}
	\usetikzlibrary{positioning}
\usepackage{tikz-cd}
\usepackage{hyperref}

\numberwithin{equation}{section}

\newcommand{\st}{:}

\let\emptyset\varnothing

\DeclareMathOperator{\End}{End}
\DeclareMathOperator{\Hom}{Hom}
\DeclareMathOperator{\Ext}{Ext}
\DeclareMathOperator{\id}{id}

\newcommand{\cc}[1]{\mathscr{C}_{#1}}

\newcommand{\sgn}{\mathsf{sign}\,}
\newcommand{\sign}[1]{\mathsf{sign}(#1)}

\newcommand{\dtmnt}{\mathsf{Det}\,}

\newcommand{\exch}[1]{\mathrm{Exch}(#1)}

\newcommand{\sfe}{\mathsf{E}}
\newcommand{\sfm}{\mathsf{M}}
\newcommand{\sfo}{\mathsf{O}}
\newcommand{\sft}{\mathsf{T}}
\newcommand{\sfu}{\mathsf{U}}
\newcommand{\sfx}{\mathsf{X}}
\newcommand{\sfy}{\mathsf{Y}}
\newcommand{\sfz}{\mathsf{Z}}

\newcommand{\triangrel}{\rightsquigarrow}

\newtheorem{theorem}{Theorem}[section]
\newtheorem{lemma}[theorem]{Lemma}
\newtheorem{proposition}[theorem]{Proposition}
\newtheorem{conjecture}[theorem]{Conjecture}
\newtheorem{corollary}[theorem]{Corollary}
\newtheorem{innercustomthm}{Theorem}
\newenvironment{customthm}[1]
  {\renewcommand\theinnercustomthm{#1}\innercustomthm}
  {\endinnercustomthm}

\theoremstyle{definition}
\newtheorem{definition}[theorem]{Definition}
\newtheorem{construction}[theorem]{Construction}
\newtheorem{example}[theorem]{Example}

\theoremstyle{remark}
\newtheorem{remark}[theorem]{Remark}

\title[Oriented matroids from type $\mathbb{A}$ cluster categories]{Oriented matroids from \\ type $\mathbb{A}$ cluster categories}
\author{Nicholas J. Williams}
\urladdr{https://nchlswllms.github.io/}
\email{nw480@cam.ac.uk}
\address{Department of Pure Mathematics and Mathematical Statistics, Centre for Mathematical Sciences, University of Cambridge, Wilberforce Rd, Cambridge, CB3 0WB,~UK}
\thanks{I would like to thank Rudradip Biswas for the opportunity to speak about a preliminary version of this work at the ``New Directions in Group Theory and Triangulated Categories'' seminar, and Mikhail Gorsky for the reference~\cite{fgppp}.
Thank you also to some referees for useful comments.
I~am currently supported by EPSRC grant EP/W001780/1.}

\begin{document}

\begin{abstract}
For any cluster-tilting object $\sft$ in the cluster category $\cc{n}$ of type~$\mathbb{A}_{n}$, we construct a rank-four oriented matroid $\mathcal{M}_{\sft}$ such that stackable triangulations of $\mathcal{M}_{\sft}$ are in bijection with equivalence classes of maximal green sequences with initial cluster~$\sft$.
This generalises the result that equivalence classes of maximal green sequences of linearly oriented $\mathbb{A}_{n}$ are in bijection with triangulations of a three-dimensional cyclic polytope.
The definition of the oriented matroid $\mathcal{M}_{\sft}$ arises from the extriangulated structure on $\cc{n}$ which makes $\sft$ projective.
\end{abstract}

\maketitle

\section{Introduction}

From the very beginning, there has been an intimate relationship between cluster algebras and the combinatorics of triangulations.
The most fundamental example of this is the $\mathbb{A}_{n}$ cluster algebra \cite{fz1}, whose cluster variables are in bijection with arcs in a convex $(n + 3)$-gon, with clusters in bijection with triangulations, and mutation corresponding to flipping an arc in a quadrilateral.
This was extended to cluster algebras from triangulations of surfaces in \cite{fst}.

Higher-even-dimensional triangulations were connected with cluster algebras in \cite{ot}, and higher-odd-dimensional ones in \cite{njw-hst}.
In particular, in \cite{njw-hst} it was shown that equivalence classes of maximal green sequences of a linearly oriented type $\mathbb{A}$ quiver were in bijection with triangulations of a three-dimensional cyclic polytope.
Maximal green sequences are finite sequences of mutations in a cluster algebra, first introduced into the mathematics literature by Keller \cite{kel-green}, and into the physics literature by Cecotti, Cordova, and Vafa \cite[Section~4.3]{ccv}.
In fact, these physicists already made a connection between maximal green sequences and three-dimensional triangulations \cite[Section~5]{ccv}.
This appearance of maximal green sequences in physics is related to their role in quantum dilogarithm identities \cite{reineke_poisson} and the Kontsevich--Soibelman wall-crossing formula from Donaldson--Thomas theory \cite{ks_stability}.

Let us explain maximal green sequences in more detail.
Choosing an initial cluster in a cluster algebra orients mutations so that some are forwards and some are backwards.
There are combinatorial rules for colouring the vertices of a quiver green or red such that forwards mutation corresponds to mutating at a green vertex, which then turns it red; backwards mutation is then the reverse of this.
In this picture, a maximal green sequence is a finite sequence of mutations at green vertices such that the quiver turns from all green to all red.
Forwards mutation may thus be called ``green mutation'', with backwards mutation called ``red''.

Given the bijection from \cite{njw-hst} between triangulations of three-dimensional cyclic polytopes and equivalence classes of maximal green sequences of linearly-oriented $\mathbb{A}_{n}$, the following natural question then arises.
Given an arbitrary orientation of the type $\mathbb{A}$ Dynkin diagram, is there a three-dimensional polytope whose triangulations are in bijection with its maximal green sequences, up to equivalence?
Following \cite{njw-hst}, this bijection should be such that the simplices in the polytope are in bijection with the mutations in the maximal green sequence, so that the number of simplices of the triangulation equals the length of the maximal green sequence.

One can deduce what properties such a three-dimensional polytope should have.
An exchange pair in a maximal green sequence corresponds to flipping an arc in a quadrilateral.
As in \cite{njw-hst}, one should lift this quadrilateral to a $3$-simplex such that the green mutation goes from the lower side of the simplex to the upper side.
These two ways of lifting a quadrilateral to a $3$-simplex are shown in Figure~\ref{fig:quad_3_simplex}.
Choosing whether a $3$-simplex in our hypothetical polytope is like the left-hand simplex in this figure or the right-hand simplex corresponds to choosing an ``orientation'' for the $3$-simplex.
This determines the ``oriented matroid'' of the hypothetical polytope.

\begin{figure}
\[
\begin{tikzpicture}[scale=1]


\begin{scope}

\coordinate(1) at (-1.5,1);
\coordinate(2) at (-2,0);
\coordinate(3) at (2,0);
\coordinate(4) at (1.5,1);

\draw (1) -- (2) -- (3) -- (4) -- (1);

\draw (1) -- (3);
\draw (2) -- (4);

\node at (1) {$\bullet$};
\node at (2) {$\bullet$};
\node at (3) {$\bullet$};
\node at (4) {$\bullet$};

\end{scope}


\begin{scope}[shift={(-4,3)}]

\coordinate(1) at (-1.5,1);
\coordinate(2) at (-2,0);
\coordinate(3) at (2,0);
\coordinate(4) at (1.5,1);

\draw (1) -- (2) -- (3) -- (4) -- (1);

\draw[dashed] (1) -- (3);
\draw (2) -- (4);

\node at (1) {$\bullet$};
\node at (2) {$\bullet$};
\node at (3) {$\bullet$};
\node at (4) {$\bullet$};

\draw[dotted] (0,-0.2) -- (2.2,-1.65);

\end{scope}


\begin{scope}[shift={(4,3)}]

\coordinate(1) at (-1.5,1);
\coordinate(2) at (-2,0);
\coordinate(3) at (2,0);
\coordinate(4) at (1.5,1);

\draw (1) -- (2) -- (3) -- (4) -- (1);

\draw (1) -- (3);
\draw[dashed] (2) -- (4);

\node at (1) {$\bullet$};
\node at (2) {$\bullet$};
\node at (3) {$\bullet$};
\node at (4) {$\bullet$};

\draw[dotted] (0,-0.2) -- (-2.2,-1.65);

\end{scope}

\end{tikzpicture}
\]
\caption{Two ways to lift a quadrilateral to a $3$-simplex}\label{fig:quad_3_simplex}
\end{figure}

In order to determine the orientation of each $3$-simplex, we need a criterion which tells us, given an initial cluster, which mutations are green.
A convenient criterion is given by \cite{fgppp} in terms of the cluster category of \cite{bmrrt}, where a choice of initial cluster is given by a choice of cluster-tilting object~$\sft$.
In particular, different orientations of $\mathbb{A}_{n}$ can be achieved by different choices of~$\sft$.
Of course, it is not obvious that orienting the $3$-simplices according to which exchange pairs are made green by the cluster-tilting object $\sft$ gives a well-defined oriented matroid $\mathcal{M}_{\sft}$: certain axioms need to be satisfied.
The first main theorem of this paper is that these axioms do in fact hold.
We moreover conjecture that $\mathcal{M}_{\sft}$ is indeed the oriented matroid of some polytope.

\begin{customthm}{A}[{Theorem~\ref{thm:matroid}}]\label{thm:int:matroids}
Given a cluster-tilting object $\sft$ in the cluster category $\cc{n}$ of type $\mathbb{A}_{n}$, we have that $\mathcal{M}_{\sft}$ is a well-defined oriented matroid.
\end{customthm}

The oriented matroids $\mathcal{M}_{\sft}$ are orientations of the uniform matroid: all possible $3$-simplices are bases of the matroid.
Orientations of the uniform matroid have been studied recently in the context of their associated ``chirotropical Grassmannians'', which were introduced in \cite{cez} in the context of scattering amplitudes, and subsequently studied in \cite{ae_tcg}.
While uniform matroids are always realisable \cite[Section~8.3]{blswz}, uniform \textit{oriented} matroids are not always realisable \cite[Proposition~8.3.2]{blswz}.

Having proved that the oriented matroids $\mathcal{M}_{\sft}$ are well-defined, we construct the desired bijection between equivalence classes of maximal green sequences with initial cluster $\sft$ and the subset of triangulations of $\mathcal{M}_{\sft}$ which are stackable.

\begin{customthm}{B}[{Theorem~\ref{thm:main}}]\label{thm:int:mgs}
Let $\sft$ be a basic cluster-tilting object in $\cc{n}$.
There is then a bijection between stackable triangulations of $\mathcal{M}_{\sft}$ and equivalence classes of maximal green sequences with initial cluster~$\sft$.
\end{customthm}

We conjecture that in fact all triangulations of $\mathcal{M}_{\sft}$ are stackable.
If this is so, and $\mathcal{M}_{\sft}$ is realisable, then for every basic cluster-tilting object $\sft \in \cc{n}$, there is a polytope whose triangulations are in bijection with equivalence classes of maximal green sequences with initial cluster~$\sft$.

This paper is structured as follows.
In Section~\ref{sect:back} we give requisite background on oriented matroids, their triangulations, cluster categories, and maximal green sequences.
In Section~\ref{sect:matroid_construction}, we construct the oriented matroids and prove Theorem~\ref{thm:int:matroids}.
In Section~\ref{sect:bijection}, we then prove the relation between their triangulations and maximal green sequences, and prove Theorem~\ref{thm:int:mgs}.

\section{Background}\label{sect:back}

We start by giving background on oriented matroids, cluster categories, and maximal green sequences.

\subsection{Oriented matroids}

A matroid is a structure abstracting the linear dependence relations that hold among a finite set of vectors.
An oriented matroid is a matroid with some extra structure --- an ``orientation'' --- which also keeps track of the signs of the coefficients in these linear relations.
We will follow the standard reference for oriented matroids \cite{blswz}.

Matroids are famous for having many different ``cryptomorphic'' definitions.
One such definition axiomatises the notion of the ``bases'' of a matroid, which correspond to bases in the standard sense of linear algebra.
Using the bases definition of a matroid, an orientation of a matroid is given by assigning a sign to every ordered basis, corresponding to the sign of the determinant of the ordered basis.
This assignment of signs gives an object called a chirotope, which we now describe.

\begin{definition}[{\cite[Definition~3.5.3]{blswz}}]\label{def:chiro}
Let $r \geqslant 1$ be an integer, and let $E$ be a finite set, known as the \emph{ground set}.
A \emph{chirotope} of \emph{rank} $r$ on $E$ is a map $\chi \colon E^r \to \{-1, 0, +1\}$ which satisfies the following three properties.
\begin{enumerate}
\item $\chi$ is not identically zero.
\item $\chi$ is alternating, that is, \[
\chi(x_{\pi(1)}, x_{\pi(2)}, \dots, x_{\pi(r)}) = \sign{\pi}\chi(x_{1}, x_{2}, \dots, x_{r})
\] for all $x_{1}, x_{2}, \dots, x_{r} \in E$ and every permutation $\pi$ of $[r] := \{1, 2, \dots, r\}$.\label{op:chiro:alt}
\item For all $x_{1}, x_{2}, \dots, x_{r}, y_{1}, y_{2}, \dots, y_{r} \in E$ such that \[
\chi(y_{i}, x_{2}, x_{3}, \dots, x_{r}) \chi(y_{1}, y_{2}, \dots, y_{i - 1}, x_{1}, y_{i + 1}, y_{i + 2}, \dots, y_{r}) \geqslant 0
\]
for all $i \in [r]$, we have that \[
\chi(x_{1}, x_{2}, \dots, x_{r}) \chi(y_{1}, y_{2}, \dots, y_{r}) \geqslant 0.
\]
\end{enumerate}

The chirotope $\chi$ is \emph{realisable} (over $\mathbb{R}$) if there exists a map $\rho \colon E \to \mathbb{R}^r$ such that \[
\chi(x_{1}, x_{2}, \dots, x_{r}) = \sgn \dtmnt (\rho(x_{1}) \rho(x_{2}) \dots \rho(x_{r})),
\] 
where here we have specified a matrix by its column vectors.

A \emph{basis} of $\chi$ is a subset $B = \{b_{1}, \dots, b_{r}\} \subseteq E$ such that $\chi(b_{1}, \dots, b_{r}) \neq 0$.
Note that by \eqref{op:chiro:alt}, this is independent of the ordering on~$B$.
A subset of $E$ is \emph{independent} if it is contained in a basis, and otherwise it is \emph{dependent}.
A dependent subset of $E$ which is minimal with respect to inclusion is called a \emph{circuit}.
\end{definition}

The set of bases of $\chi$ constitutes the \emph{underlying unoriented matroid}.
If the set of bases is $\binom{E}{r} := \{B \subseteq E \st |B| = r\}$, then the underlying unoriented matroid is the \emph{uniform matroid}.
This will be the case for the matroids that we study in this paper.

The chirotope of an oriented matroid can be used to give the circuits signatures in the following sense.
We do not give detail on how the signature of a circuit is derived from the chirotope; the reference is \cite[Lemma~3.5.7]{blswz}.

\begin{definition}[{\cite[p.102]{blswz}}]
A \emph{signed set} $X$ is a set $\underline{X}$ together with a partition $(X^+, X^-)$ of $\underline{X}$ into two distinguished subsets.
Here $\underline{X}$ is called the \emph{support} of~$X$.
We call $(X^{+}, X^{-})$ a \emph{signature} of~$X$.
Given a set $E$, a \emph{signed subset} of $E$ is a signed set $X$ with $\underline{X} \subseteq E$.
\end{definition}

One can specify axioms for a set of signed subsets of the ground set $E$ to be the set of signed circuits of an oriented matroid \cite[Definition~3.2.1]{blswz}.
Note that we do not identify an oriented matroid with its chirotope or with its set of signed circuits.
We think of oriented matroids as objects in the style of object-oriented programming: an oriented matroid has a chirotope and a set of signed circuits as attributes; it is determined by specifying either of these pieces of data, but it is not identified with either of them.
We say that an oriented matroid is \emph{realisable} (over $\mathbb{R}$) if its chirotope is realisable (over $\mathbb{R}$).
The following example of an oriented matroid will occur in a couple of places later in the paper.

\begin{example}\label{ex:cyclic_polytope}
We denote $[m] := \{1, 2, \dots, m\}$.
There is a rank-four oriented matroid on $[m]$ defined by the chirotope \[
\chi(a, b, c, d) = +1 \quad \text{ if } a < b < c < d,
\]
whose value on quadruples which are not totally ordered is determined by the requirement that $\chi$ is alternating.
This is the oriented matroid of a three-dimensional cyclic polytope with $m$ vertices and can be realised using Vandermonde matrices, see \cite[Section~9.4]{blswz} for more details.
The signed circuits of this oriented matroid are given by $(\{a, c, e\}, \{b, d\})$ and $(\{b, d\}, \{a, c, e\})$ for $a < b < c < d < e$ \cite{breen}.
\end{example}

In this paper, it will be useful to restrict an oriented matroid to a subset.

\begin{proposition}[{\cite[Proposition~3.3.1, pp.133--4]{blswz}}]\label{prop:matroid_restriction}
Let $\mathcal{M}$ be an oriented matroid on a ground set $E$ with set of signed circuits $\mathcal{C}$ and let $V$ be a subset of $E$.
Then $\mathcal{C}(V) := \{C \in \mathcal{C} \st \underline{C} \subseteq V\}$, the set of circuits of $\mathcal{M}$ contained in $V$, is the set of signed circuits of an oriented matroid on $V$.
We call this oriented matroid the \emph{restriction} of $\mathcal{M}$ to $V$, and denote it by $\mathcal{M}(V)$.

Let $\chi \colon E^{r} \to \{-1, 0, +1\}$ be the chirotope of $\mathcal{M}$.
If $\mathcal{M}(V)$ also has rank $r$, then its chirotope is $\chi|_{V^r}$.
\end{proposition}

For every oriented matroid $\mathcal{M}$ on a set $E$, there is a dual oriented matroid $\mathcal{M}^{\ast}$ on $E$ \cite[Proposition~3.4.1]{blswz}.
The (unsigned) circuits of $\mathcal{M}^{\ast}$ are the subsets $Y \subseteq E$ which are inclusion-minimal with respect to the property that $|Y \cap X| \neq 1$ for all circuits $X$ of~$\mathcal{M}$.
The circuits of $\mathcal{M}^{\ast}$ are known as the \emph{cocircuits} of~$\mathcal{M}$.
Lemma~\ref{lem:chi->cocirc}, which we will see later, tells us how the signatures of the cocircuits can be derived from the chirotope.

If $\mathcal{M}$ is realisable, then linear hyperplanes give cocircuits of $\mathcal{M}$: one half of the cocircuit is the subset of $E$ lying on one side of the hyperplane, the other half is the subset lying on the other \cite[p.116]{blswz}, and the complement of the cocircuit is the subset lying in the hyperplane.

\begin{definition}[{\cite[p.377]{blswz}}]\label{def:hyp+facet}
A signed circuit $X$ is \emph{positive} if $\underline{X} = X^{+}$. 
An oriented matroid $\mathcal{M}$ is \emph{acyclic} if it does not have a positive circuit.

Let $\mathcal{M}$ be an acyclic oriented matroid on a set~$E$.
A \emph{hyperplane} of $\mathcal{M}$ is a subset $H \subseteq E$ such that $E \setminus H$ is a cocircuit of~$\mathcal{M}$.
An \emph{open halfspace} of $\mathcal{M}$ is $Y^+$ for $Y$ a cocircuit of $\mathcal{M}$.
A \emph{facet} of $\mathcal{M}$ is a hyperplane $H$ such that $E \setminus H$ is an open halfspace.
\end{definition}

Finally, we define triangulations of oriented matroids.

\begin{definition}[{\cite[Theorem~2.4(f)]{santos_tom}}]\label{def:mat_triang}
Let $\mathcal{M}$ be an acyclic oriented matroid of rank $r$ on a set $E$.
A non-empty collection $\Delta$ of bases of $\mathcal{M}$ is called a \emph{triangulation} of $\mathcal{M}$ if the following conditions hold.
\begin{enumerate}
\item If $\sigma \in \Delta$, then each facet $\tau$ of $\mathcal{M}(\sigma)$ is either contained in a facet of $\mathcal{M}$ or there exists $\sigma' \in \Delta \setminus \{\sigma\}$ such that $\tau \subseteq \sigma'$.\label{op:mat_triang:facet}
\item No two bases $\sigma$ and $\tau$ \emph{overlap on a circuit} $C = (C^{+}, C^{-})$, meaning that $C^{+} \subseteq \sigma$ and there is an element $a \in C^{+}$ such that $\underline{C} \setminus \{a\} \subseteq \tau$.\label{op:mat_triang:int}
\end{enumerate}
We also refer to the bases in $\Delta$ as \emph{$(r - 1)$-simplices}.
\end{definition}

\subsection{Cluster categories of type~$\mathbb{A}$}\label{sect:back:cca}

Cluster categories were introduced in the paper \cite{bmrrt}, which we refer to for full background.
In type~$\mathbb{A}$, these categories were first introduced in \cite{ccs}.
Given an acyclic quiver~$Q$, its path algebra over a field $K$ is denoted $KQ$.
We write $\mathscr{D}^{b}(KQ)$ for the bounded derived category of right $KQ$-modules which are finite-dimensional over~$K$.
The \emph{cluster category} $\cc{Q}$ of $Q$ is the orbit category $\mathscr{D}^{b}(KQ)/\nu\Sigma^{-2}$, where $\nu := D(KQ) \otimes^{\mathbf{L}}_{KQ}-\colon \mathscr{D}^{b}(KQ) \to \mathscr{D}^{b}(KQ)$ is the Nakayama functor and $\Sigma$ is the suspension of $\mathscr{D}^{b}(KQ)$.
Here $D(-) = \Hom_{K}(-, K)$ is the standard duality.

The linearly oriented $\mathbb{A}_{n}$ quiver is $1 \leftarrow 2 \leftarrow \dots \leftarrow n$ and we denote its cluster category by $\cc{n}$.
For technical reasons, we prefer to have the arrows of this quiver going in this direction rather than to the right, which is more usual.
(However, $\cc{n}$ is independent of the choice of orientation of the quiver.)
The cluster category $\cc{n}$ is a $K$-linear triangulated category with finitely many indecomposable objects, and one can describe its morphisms and suspension functor combinatorially.
Indeed, up to isomorphism, the indecomposable objects of $\cc{n}$ may be labelled $\sfo_{B}$ for $B = \{b_{1}, b_{2}\} \in \binom{[n + 3]}{2}$ with $b_{1} \neq b_{2} \pm 1$.
Here, when we write $b_{2} \pm 1$, we are using modular arithmetic with representatives in $[n + 3]$, rather than $\{0, 1, \dots, n + 2\}$, as is more usual; we will continue do this in this paper.
Hence, indecomposable objects in $\cc{n}$ are in bijection with arcs in a convex $(n + 3)$-gon.

We call a permutation in the subgroup of the symmetric group $\mathfrak{S}_{m}$ generated by the permutation $x \mapsto x + 1$ a \emph{cyclic} permutation.
By \emph{rotation}, we mean applying a cyclic permutation.
Given $\{a_{1}, a_{2}, \dots, a_{l}\} \subseteq [m]$ with $l \geqslant 3$, we write $a_{1} \prec a_{2} \prec \dots \prec a_{l}$ if there is a cyclic permutation $\pi \in \mathfrak{S}_{m}$ such that $\pi(a_{1}) < \pi(a_{2}) < \dots < \pi(a_{l})$ in the usual ordering on $[m]$.
We call the ordering $a_{1} \prec a_{2} \prec \dots \prec a_{l}$ a \emph{cyclic ordering}. 
We use the symbol $\preccurlyeq$ with the obvious meaning: either the elements either side of $\preccurlyeq$ are equal, or the cyclic ordering holds.
We then have that $\Hom_{\cc{n}}(\sfo_{B}, \sfo_{C}) \neq 0$ if and only if $b_{1} - 1 \prec c_{1} \prec b_{2} - 1 \prec c_{2}$ with $B = \{b_{1}, b_{2}\}$ and $C = \{c_{1}, c_{2}\}$.
Moreover, we have that if $\Hom_{\cc{n}}(\sfo_{B}, \sfo_{C}) \neq 0$, then $\Hom_{\cc{n}}(\sfo_{B}, \sfo_{C}) \cong K$.
Finally, we have $\Sigma\sfo_{B} = \sfo_{B - 1}$, where $B \pm 1 := \{b \pm 1 \st b \in B\}$ and $\Sigma$ likewise denotes the suspension functor in the category $\cc{n}$.
These descriptions follow, for instance, from specialising the results of \cite{ot}, but we will also explain them more in the appendix.

An object $\sft$ of $\cc{n}$ is \emph{cluster-tilting} if $\Hom_{\cc{n}}(\sft, \Sigma\sft) = 0$ and it has $n$ isomorphism classes of indecomposable summands.
There is a bijection between basic cluster-tilting objects in $\cc{n}$ and triangulations of the $(n + 3)$-gon via the map sending a basic cluster-tilting object $\sft = \bigoplus_{i = 1}^n \sfo_{B_i}$ to the triangulation with set of arcs $\{B_{1}, B_{2}, \dots, B_{n}\}$.
This is because the condition of $\sft$ being cluster-tilting is equivalent to the condition that none of the arcs $B_{i}$ cross.
Recall that \emph{basic} means that distinct summands are non-isomorphic.

\subsection{Maximal green sequences}

To suit our purposes, we use a different presentation of maximal green sequences to the usual one.
Our presentation uses the concept of an extriangulated category, for background on which, see \cite{np}.
Given a cluster-tilting object $\sft$ in $\cc{n}$, we consider the extriangulated structure on $\cc{n}$ given by those triangles $\sfx \to \sfy \to \sfz \to \Sigma\sfx$ for which the morphism $\sfz \to \Sigma\sfx$ factors through~$\Sigma\sft$.
We write $\Ext_{\sft}^{1}(\sfz, \sfx)$ for $\Ext$ in this extriangulated structure.
This extriangulated structure was first introduced in \cite[Section~4.7.1]{pppp} and is an example of a relative extriangulated structure as studied in \cite[Section~3.2]{hln}.

We say that two basic cluster tilting objects $\sft_{1}$ and $\sft_{2}$ of $\cc{n}$ are related by a \emph{mutation} if $\sft_{1} = \sfe \oplus \sfx$ and $\sft_{2} = \sfe \oplus \sfy$, where $\sfx$ and $\sfy$ are indecomposable and not isomorphic.
We say that the mutation from $\sft_{1}$ to $\sft_{2}$ is \emph{$\sft$-green} if $\Ext_{\sft}^{1}(\sfy, \sfx) \neq 0$.
We then say that $(\sfx, \sfy)$ is the \emph{exchange pair} of the mutation.

A \emph{maximal green sequence} of $\sft$ is a sequence $\mathcal{G} = (\sft_{0}, \sft_{1}, \dots, \sft_{l})$ of basic cluster-tilting objects $\sft_{i} \in \cc{n}$ such that $\sft_{0} = \sft$, $\sft_{l} = \Sigma\sft$, and for all $i \in [l]$, $\sft_{i}$ is a $\sft$-green mutation of $\sft_{i - 1}$.
Denote by $(\sfx_{i}, \sfy_{i})$ the exchange pair of the $\sft$-green mutation from $\sft_{i - 1}$ to $\sft_{i}$.
We say that $\sfx_{i}$ \emph{enters} $\mathcal{G}$ at $(\sfx_{i}, \sfy_{i})$ and $\sfy_{i}$ \emph{leaves}~$\mathcal{G}$.
We write $\exch{\mathcal{G}} := \{(\sfx_{1}, \sfy_{1}), \dots, (\sfx_{l}, \sfy_{l})\}$
for the set of exchange pairs of $\mathcal{G}$.
We say that two maximal green sequences $\mathcal{G}$ and $\mathcal{G}'$ of $\sft$ are \emph{equivalent} if $\exch{\mathcal{G}} = \exch{\mathcal{G}'}$.
See \cite[Theorem~3.19]{njw-mg} for other characterisations of this equivalence relation.

\begin{remark}
What we call a maximal green sequence of $\sft$ corresponds to a maximal green sequence of $\End_{\cc{n}}\sft$ in the literature.
We explain how our definition is equivalent to the definition of maximal green sequence using two-term silting complexes (see \cite[Section~2.3.5]{njw-mg}).

It follows from \cite[Theorem~6.15]{fgppp} that the category $K^{[-1,0]}(\mathrm{proj}\, \End_{\cc{n}}\sft)$ of two-term complexes of projectives over $\End_{\cc{n}}\sft$ is equivalent, as an extriangulated category, to the extriangulated ideal quotient of $\cc{n}$ (with the above extriangulated structure) by the morphisms from the shifted projectives to the projectives.
It follows from \cite[Proposition~5.1]{fgppp} that the relative extriangulated structure used in \cite[Theorem~6.15]{fgppp} is the one described above.
Moreover, this quotient induces a bijection between basic cluster-tilting objects in $\cc{n}$ and two-term silting complexes in $K^{[-1,0]}(\mathrm{proj}\, \End_{\cc{n}}\sft)$ \cite[Remark~6.20]{fgppp}.

The definition of a maximal green sequence of $\End_{\cc{n}}\sft$ in terms of two-term silting complexes is a sequence of green mutations of basic two-term silting complexes from the projectives to the shifted projectives.
By the equivalence of extriangulated categories from the previous paragraph, this corresponds to our definition of a maximal green sequence of $\sft$ as a sequence of $\sft$-green mutations from $\sft$ to $\Sigma\sft$.
\end{remark}

\section{Constructing the oriented matroids}\label{sect:matroid_construction}

As in Section~\ref{sect:back:cca}, a cluster-tilting object in $\cc{n}$ corresponds to a triangulation of a convex $(n + 3)$-gon.
In this section, we will construct an oriented matroid $\mathcal{M}_{T}$ from a triangulation $T$ of a convex $(n + 3)$-gon.
In Section~\ref{sect:bijection}, we will show that if $\Sigma\sft$ is the cluster-tilting object in $\cc{n}$ corresponding to $T$, then $\mathcal{M}_{\sft} := \mathcal{M}_{T}$ is the oriented matroid with the desired properties.
That is, we will show that stackable triangulations of $\mathcal{M}_{\sft}$ are in bijection with equivalence classes of maximal green sequences of~$\sft$.

For now, let $T$ be a triangulation of a convex $m$-gon, where $m \geqslant 4$.
We label the vertices of the $m$-gon anti-clockwise by the elements of $[m]$.
We construct an oriented matroid $\mathcal{M}_{T}$ from $T$ by specifying its chirotope.
Given $x, y \in [m]$ with $x \neq y$, we denote the arc with vertices $x$ and $y$ by $xy$.
We refer to a subset of $[m]$ of size $k$ as a \emph{$k$-subset}.

\begin{definition}\label{def:mt}
The ground set of the oriented matroid $\mathcal{M}_{T}$ is $[m]$.
Let $\{a, b, c, d\}$ be a 4-subset of~$[m]$ such that $a \prec b \prec c \prec d$.
We then define \[
\chi_{T}(a, b, c, d) := \left\{
	\begin{array}{ll}
	+1 & \text{ if there is an arc } xy \text{ of } T \text{ with } a \prec b \preccurlyeq x \prec c \prec d \preccurlyeq y, \\
	-1 & \text{ if there is an arc } xy \text{ of } T \text{ with } a \preccurlyeq x \prec b \prec c \preccurlyeq y \prec d.
	\end{array}
 \right.
\]
We extend this to tuples which are not cyclically ordered by requiring that $\chi_{T}$ be alternating, in the sense of Definition~\ref{def:chiro}\eqref{op:chiro:alt}.
\end{definition}

We first need to show that $\chi_{T}$ is well-defined as a map.

\begin{lemma}
Given $\{a \prec b \prec c \prec d\} \subseteq [m]$, the triangulation $T$ cannot both contain an arc $xy$ with $a \prec b \preccurlyeq x \prec c \prec d \preccurlyeq y$ and an arc $x'y'$ with $a \preccurlyeq x' \prec b \prec c \preccurlyeq y' \prec d$.
\end{lemma}
\begin{proof}
If such arcs $xy$ and $x'y'$ exist, then we have $a \preccurlyeq x' \prec b \preccurlyeq x \prec c \preccurlyeq y' \prec d \preccurlyeq y$, and so we have $x' \prec x \prec y' \prec y$.
This means that the arcs $xy$ and $x'y'$ cross, which contradicts the fact that $T$ is a triangulation.
\end{proof}

\begin{lemma}
Given $\{a \prec b \prec c \prec d\} \subseteq [m]$, the triangulation $T$ either contains an arc $xy$ with $a \prec b \preccurlyeq x \prec c \prec d \preccurlyeq y$ or an arc $xy$ with $a \preccurlyeq x \prec b \prec c \preccurlyeq y \prec d$.
\end{lemma}
\begin{proof}
If the triangulation $T$ contains arcs $xy$ with $a \preccurlyeq x \prec y \prec b$, then $T$ restricts to a triangulation of the polygon with vertices $V = [m] \setminus \{v \in [m] \st x \prec v \prec y\}$, where we also have $V \supseteq \{a, b, c, d\}$.
Hence, we can assume that there are no such arcs $xy$ of $T$, and likewise assume that $T$ contains no arcs $xy$ with $b \preccurlyeq x \prec y \prec c$, $c \preccurlyeq x \prec y \prec d$, or $d \preccurlyeq x \prec y \prec a$.
Thus, if $x$ is a vertex with $a \preccurlyeq x \prec b$, then there must be an arc (or boundary edge) $xy$ of $T$ with either $b \preccurlyeq y \prec c$, $c \preccurlyeq y \prec d$, or $d \preccurlyeq y \prec a$.
If the middle case holds, then we are done, so we can assume ($\ast$) that the first or the third case holds for all $x$ with $a \preccurlyeq x \prec b$.
If we have $a \preccurlyeq x \prec x' \prec b$, with $x'y'$ an arc of~$T$, then we cannot have $b \preccurlyeq y \prec c$ and $d \preccurlyeq y' \prec a$, since then $xy$ and $x'y'$ cross.
Thus, let $\overrightarrow{x}$ be the maximal vertex with $a \preccurlyeq \overrightarrow{x} \prec b$ such that there is an arc $\overrightarrow{x}\overrightarrow{y}$ of $T$ with $d \preccurlyeq \overrightarrow{y} \prec a$.
Similarly, let $\overleftarrow{x}$ be the minimal vertex with $a \preccurlyeq \overleftarrow{x} \prec b$ such that there is an arc $\overleftarrow{x}\overleftarrow{y}$ of $T$ with $b \preccurlyeq \overleftarrow{y} \prec c$.
By assumption~($\ast$), there are no vertices $x$ with $\overrightarrow{x} \prec x \prec \overleftarrow{x}$.
In fact, one can see that $\overrightarrow{x} = \overleftarrow{x}$, since one can otherwise consider the triangle with the side $\overrightarrow{x}\overleftarrow{x}$, which must either contain an arc $\overrightarrow{x}y$ with $b \preccurlyeq y \prec c$ or an arc $\overleftarrow{x}y$ with $d \preccurlyeq y \prec a$.

We thus let $x_{ab} = \overrightarrow{x} = \overleftarrow{x}$.
One can deduce the existence of analogous vertices $x_{bc}$, $x_{cd}$, and $x_{da}$ with $b \preccurlyeq x_{bc} \prec c$, $c \preccurlyeq x_{cd} \prec d$, and $d \preccurlyeq x_{da} \prec a$.
We have that there are arcs $\overleftarrow{y}x_{ab}$ and $x_{ab}\overrightarrow{y}$ with $d \preccurlyeq \overleftarrow{y} \prec a$ and $b \preccurlyeq \overrightarrow{y} \prec a$, and analogous arcs for each of $x_{bc}$, $x_{cd}$, and $x_{da}$.
If none of these arcs are to intersect, we must have that they form a quadrilateral with vertices $x_{ab}$, $x_{bc}$, $x_{cd}$, and $x_{da}$, so that in particular $\overleftarrow{y} = x_{da}$ and $\overrightarrow{y} = x_{bc}$.
The diagonal $xy$ of this quadrilateral which is an arc of $T$ then gives us the arc in the statement of the lemma.
\end{proof}

We now need to show that $\mathcal{M}_{T}$ is a well-defined oriented matroid.
We will use \cite[Corollary~3.6.3]{blswz}, which says that $\mathcal{M}_{T}$ is well-defined if $\chi_{T}$ is alternating, the underlying (unoriented) matroid of $\mathcal{M}_{T}$ is well-defined, and every six-element restriction of $\mathcal{M}_{T}$ is realisable.
The first property we have required to be true by construction.
The second property is immediate, since the underlying matroid of $\mathcal{M}_{T}$ is the uniform matroid.
To show the final property, we need to understand restrictions of $\mathcal{M}_{T}$.

\begin{definition}
Let $T$ be a triangulation of $[m]$ and let $V \subseteq [m]$ be a subset with $|V| \geqslant 4$.
We say that an arc $xy$ of $T$ \emph{separates} $V$ if $V = V_{x} \sqcup V_{y}$ with $V_{x}$ and $V_{y}$ non-empty, with $V_{x} \preccurlyeq x \prec V_{y} \preccurlyeq y$.
By this, we of course mean that for every $a \in V_{x}$ and $c \in V_{y}$, we have $a \preccurlyeq x \prec c \preccurlyeq y$.

Given an arc $xy$ which separates $V$, set $x^{V} := \max_{\circlearrowleft} V_{x}$ and $y^{V} := \max_{\circlearrowleft} V_{y}$, where $\max_{\circlearrowleft} V_{x}$ is defined to be the unique element of $V_{x}$ such that $v \preccurlyeq \max_{\circlearrowleft} V_{x} \preccurlyeq x$ for all elements $v \in V_{x}$, and $\max_{\circlearrowleft} V_{y}$ is defined similarly.
Define $T_{V}$ to be the triangulation of the polygon with vertex set $V$ whose set of arcs is
\begin{align*}
&\{x^{V}y^{V} \st xy \text{ an arc of $T$ which separates }V\} \\ &\qquad \setminus \{\text{boundary arcs of the polygon with vertex set }V\}.
\end{align*}
\end{definition}

\begin{lemma}
Given a triangulation $T$ of the $m$-gon and $V \subseteq [m]$ with $|V| \geqslant 4$, we have that $T_{V}$ is a triangulation of the convex polygon with vertices $V$.
\end{lemma}
\begin{proof}
Note first that if $|V| = m$, then the statement is immediate, since in that case all arcs of $T$ separate $V$.
It suffices then to consider the case where $|V| = m - 1$, since one can use induction to obtain the statement for smaller $V$.
We can then rotate to assume that $V = [m] \setminus \{m\}$.
In this case, $T_{V}$ has arcs 
\begin{align*}
&(\{xy \text{ an arc of } T \st x, y \in [m - 1]\} \cup \{x(m - 1) \st xm \text{ an arc of } T\}) \\ &\qquad \setminus \{\text{boundary arcs of the polygon with vertices in }V\}.
\end{align*}
The fact that this is a triangulation of the convex polygon with vertices~$V$ then follows from \cite[Theorem~4.2(iii)]{rambau}, but is also not hard to see by inspection.
\end{proof}

\begin{proposition}\label{prop:restriction}
Let $T$ be a triangulation of the $m$-gon and $V \subseteq [m]$ with $|V| \geqslant 4$.
Then $\chi_{T}|_{V^{4}} = \chi_{T_{V}}$, and so $\mathcal{M}_{T}(V) = \mathcal{M}_{T_{V}}$.
\end{proposition}
\begin{proof}
Let $\{a \prec b \prec c \prec d\} \subseteq V$.
Then $\chi_{T}(a, b, c, d) = + 1$ if and only if there is an arc $xy$ of $T$ with $a \prec b \preccurlyeq x \prec c \prec d \preccurlyeq y$.
Note that $xy$ therefore separates~$V$.
This is then the case if and only if $a \prec b \preccurlyeq x^V \prec c \prec d \preccurlyeq y^V$.
Finally, this is the case if and only if there exists an arc $x^{V}y^{V}$ of $T_{V}$ such that this holds, which is the condition for $\chi_{T_{V}}(a, b, c, d) = + 1$.
\end{proof}

Hence, using \cite[Corollary~3.6.3]{blswz}, in order to show that $\chi_{T}$ is a well-defined chirotope, it suffices to show that $\chi_{T}$ is realisable for $T$ a triangulation of a hexagon.

\begin{proposition}\label{prop:hex_realisable}
If $T$ is a hexagon triangulation, then $\chi_{T}$ is a realisable chirotope.
\end{proposition}
\begin{proof}
If $T$ is the triangulation $\{26, 36, 46\}$, then one can check that $\chi_{T}$ is positive on all 4-subsets of $[6]$ ordered in the usual way.
Hence, $\chi_{T}$ is realised by the three-dimensional cyclic polytope with six vertices from Example~\ref{ex:cyclic_polytope}.
By rotating, we have that $\chi_{T}$ is realisable for any triangulation $T$ of the form $\{(a - 4)a, (a - 3)a, (a - 2)a\}$.

Up to rotation, there are only three more hexagon triangulations we need to check, namely $T_{\Delta}$, $T_{\scriptsize\reflectbox{N}}$, and $T_{N}$ as shown in Figure~\ref{fig:hex_triangs}.
We claim that $\mathcal{M}_{T_{\Delta}}$, $\mathcal{M}_{T_{\tiny\reflectbox{N}}}$, and $\mathcal{M}_{T_{N}}$ are realised by the respective matrices \[
M_{T_{\Delta}} =
\left(
\begin{smallmatrix}
1 & 1 & 1 & 1 & 1 & 1 \\
0 & 1 & 2 & 3 & 4 & 5 \\
0 & 1 & 4 & 9 & 16 & 25 \\
3 & 1 & -1 & 1 & 4 & 12 \\
\end{smallmatrix}
\right),
\,
M_{T_{\tiny\reflectbox{N}}} = 
\left(
\begin{smallmatrix}
1 & 1 & 1 & 1 & 1 & 1 \\
0 & 1 & 2 & 3 & 4 & 5 \\
0 & 1 & 4 & 9 & 16 & 25 \\
10 & 1 & -1 & 1 & 11 & 32
\end{smallmatrix}
\right),
\,
M_{T_{N}} = 
\left(
\begin{smallmatrix}
1 & 1 & 1 & 1 & 1 & 1 \\
0 & 1 & 2 & 3 & 4 & 5 \\
0 & 1 & 4 & 9 & 16 & 25 \\
4 & 1 & -1 & 1 & 11 & 26
\end{smallmatrix}
\right).
\]
Indeed, the maximal minors of these matrices are as shown in Table~\ref{tab:minors}.
We leave it to the reader to check that the signs of these minors agree with Definition~\ref{def:mt}.
\end{proof}

\begin{figure}
\[
\begin{tikzpicture}[scale=1.3]

\begin{scope}[shift={(-3.5,0)}]

\coordinate(1) at (120:1);
\coordinate(2) at (180:1);
\coordinate(3) at (240:1);
\coordinate(4) at (300:1);
\coordinate(5) at (360:1);
\coordinate(6) at (60:1);

\draw (1) -- (2) -- (3) -- (4) -- (5) -- (6) -- (1);

\node at (120:1.2) {$1$};
\node at (180:1.2) {$2$};
\node at (240:1.2) {$3$};
\node at (300:1.2) {$4$};
\node at (360:1.2) {$5$};
\node at (60:1.2) {$6$};

\draw (2) -- (6);
\draw (2) -- (4);
\draw (4) -- (6);

\node at (270:1.5) {$T_{\Delta}$};

\end{scope}

\begin{scope}[shift={(0,0)}]

\coordinate(1) at (120:1);
\coordinate(2) at (180:1);
\coordinate(3) at (240:1);
\coordinate(4) at (300:1);
\coordinate(5) at (360:1);
\coordinate(6) at (60:1);

\draw (1) -- (2) -- (3) -- (4) -- (5) -- (6) -- (1);

\node at (120:1.2) {$1$};
\node at (180:1.2) {$2$};
\node at (240:1.2) {$3$};
\node at (300:1.2) {$4$};
\node at (360:1.2) {$5$};
\node at (60:1.2) {$6$};

\draw (3) -- (6);
\draw (4) -- (6);
\draw (1) -- (3);

\node at (270:1.5) {$T_{\scriptsize\reflectbox{N}}$};

\end{scope}

\begin{scope}[shift={(3.5,0)}]

\coordinate(1) at (120:1);
\coordinate(2) at (180:1);
\coordinate(3) at (240:1);
\coordinate(4) at (300:1);
\coordinate(5) at (360:1);
\coordinate(6) at (60:1);

\draw (1) -- (2) -- (3) -- (4) -- (5) -- (6) -- (1);

\node at (120:1.2) {$1$};
\node at (180:1.2) {$2$};
\node at (240:1.2) {$3$};
\node at (300:1.2) {$4$};
\node at (360:1.2) {$5$};
\node at (60:1.2) {$6$};

\draw (3) -- (6);
\draw (2) -- (6);
\draw (3) -- (5);

\node at (270:1.5) {$T_{N}$};

\end{scope}

\end{tikzpicture}
\]
\caption{Three hexagon triangulations $T$ for which it suffices to check that $\mathcal{M}_{T}$ is realisable}\label{fig:hex_triangs}
\end{figure}

\begin{table}
\[
\begin{tabular}{c|c|c|c}
Column tuple & Minor of $M_{T_{\Delta}}$ & Minor of $M_{T_{\tiny\reflectbox{N}}}$ & Minor of $M_{T_{N}}$ \\
\hline
1234 & 8 & $-6$ & 6 \\
1235 & 18 & $-10$ & 26 \\
1236 & 38 & $-6$ & 54 \\
1245 & 6 & 6 & 42 \\
1246 & 34 & 42 & 102 \\
1256 & 48 & 64 & 64 \\
1345 & $-10$ & 18 & 30 \\
1346 & $-6$ & 72 & 72 \\
1356 & 34 & 102 & 42 \\
1456 & 38 & 54 & $-6$ \\
2345 & $-6$ & 8 & 8 \\
2346 & $-10$ & 30 & 18 \\
2356 & 6 & 42 & 6 \\
2456 & 18 & 26 & $-10$ \\
3456 & 8 & 6 & $-6$ \\
\end{tabular}
\]
\caption{Minors of the matrices $M_{T_{\Delta}}$, $M_{T_{\tiny\reflectbox{N}}}$, and $M_{T_{N}}$}\label{tab:minors}
\end{table}

Note that Table~\ref{tab:minors} gives examples of the the oriented matroids defined in Definition~\ref{def:mt}, via taking the signs of the minors.
We then obtain the following result from Proposition~\ref{prop:hex_realisable} by applying \cite[Corollary~3.6.3]{blswz}.

\begin{theorem}\label{thm:matroid}
We have that $\chi_{T}$ is a well-defined chirotope, and so $\mathcal{M}_{T}$ is a well-defined oriented matroid.
\end{theorem}

It will be useful to describe the signed circuits of $\mathcal{M}_{T}$.

\begin{proposition}\label{prop:circuits}
Every signed circuit of $\mathcal{M}_{T}$ is of the form $C = (\{a, c, e\}, \{b, d\})$ or $C = (\{b, d\}, \{a, c, e\})$ for $a \prec b \prec c \prec d \prec e$.
\end{proposition}
\begin{proof}
First, note that every 5-subset of $[m]$ is a circuit of $\mathcal{M}_{T}$, since every 4-subset is a basis.
Thus, let $\underline{C} = \{a \prec b \prec c \prec d \prec e\} \subseteq [m]$ be the underlying set of a signed circuit of $\mathcal{M}_{T}$, where $a \prec b \prec c \prec d \prec e$.
The signs of $C$ are determined by the restriction $\mathcal{M}_{T}(\underline{C}) = \mathcal{M}_{T_{\underline{C}}}$ by Proposition~\ref{prop:matroid_restriction}.
We have that $T_{\underline{C}}$ is a triangulation of a pentagon.
Up to rotation, there is only one triangulation of a pentagon, so we can rotate and assume that $T_{\underline{C}} = \{be, ce\}$.
We then have that $\mathcal{M}_{T_{\underline{C}}}$ is then the oriented matroid of a three-dimensional cyclic polytope with five vertices.
This implies that $C = (\{a, c, e\}, \{b, d\})$ or $C = (\{b, d\}, \{a, c, e\})$, as in Example~\ref{ex:cyclic_polytope}.
\end{proof}

\begin{corollary}
$\mathcal{M}_{T}$ is acyclic.
\end{corollary}

\begin{remark}
It follows from the fact that $\mathcal{M}_{T}$ is acyclic that if it is realisable as a configuration of vectors $\mathbb{R}^{4}$, then it is realisable as an affine configuration of points in $\mathbb{R}^{3}$ (obtained by scaling the vectors so that the first coordinates equal~$1$).
Moreover, since $\mathcal{M}_{T}$ has no signed circuit where one half is a singleton, this affine configuration of points will be the vertex set of a convex polytope.
(See, for instance, \cite[pp.13-14]{anderson} for these facts.)
Hence, $\mathcal{M}_{T}$ is realisable over $\mathbb{R}$ if and only if it is realisable as the set of vertices of a convex polytope.
\end{remark}

Since $\mathcal{M}_{T}$ is acyclic, we may talk about the facets of the matroids $\mathcal{M}_{T}$.
We divide these into two types: upper facets and lower facets.
This follows standard practice for cyclic polytopes; see \cite{er}.

\begin{definition}
A facet $abc$ is an \emph{upper facet} of $\mathcal{M}_{T}$ if and only if $\chi_{T}(x, a, b, c) = +1$ for all $x \in [m] \setminus \{a, b, c\}$.
Similarly, a facet $abc$ is a \emph{lower facet} of $\mathcal{M}_{T}$ if and only if $\chi_{T}(x, a, b, c) = -1$ for all $x \in [m] \setminus \{a, b, c\}$.
\end{definition}

It follows from Lemma~\ref{lem:chi->cocirc} that all facets of $\mathcal{M}_{T}$ are either upper facets or lower facets.
We can then describe the upper and lower facets of $\mathcal{M}_{T}$ as follows.
We write $T \pm 1$ for the triangulation of the $m$-gon with triangles $H \pm 1$ where $H$ is a triangle of $T$.

\begin{proposition}\label{prop:facets}
The upper facets of the oriented matroid $\mathcal{M}_{T}$ are given by $H$ where $H$ is a triangle of $T$, and the lower facets are given by $H$ where $H$ is a triangle of $T + 1$.
\end{proposition}

In order to prove Proposition~\ref{prop:facets}, we will need to use \cite[Lemma~3.5.8]{blswz}, which we reproduce for the convenience of the reader.
This lemma tells us how the chirotope determines the signatures of the cocircuits.

\begin{lemma}[{\cite[Lemma~3.5.8]{blswz}}]\label{lem:chi->cocirc}
Let $\chi$ be a chirotope on a set $E$ with a corresponding oriented matroid $\mathcal{M}$.
Let $\underline{D}$ be the underlying set of a signed cocircuit of~$\mathcal{M}$.
Given $e, f \in \underline{D}$ with $e \neq f$, set \[
s(e, f) = \chi(e, x_{2}, \dots, x_{r}) \cdot \chi(f, x_{2}, \dots, x_{r}) \in \{-1,+1\}
\]
where $X = (x_{2}, \dots, x_{r})$ is an ordered basis of the hyperplane $E \setminus \underline{D}$ of $\mathcal{M}$.
(Here, a basis of the hyperplane is a maximally independent set contained in the hyperplane.)
Then $s(e, f)$ does not depend upon the choice of $X$.

The signatures of $\underline{D}$ are given by $(D^+, D^-)$ and $(D^-, D^+)$, where
\begin{align*}
D^+ &= \{e\} \cup \{f \in \underline{D} \setminus e \st s(e, f) = +1\},\\
D^- &= \{f \in \underline{D} \setminus e \st s(e, f) = -1\}.
\end{align*}
Moreover, this pair of signatures does not depend upon the choice of~$e$.
\end{lemma}

In our case, the underlying unoriented matroid of $\chi_{T}$ is the rank-four uniform matroid on $[m]$, so the bases are precisely the 4-subsets of $[m]$.
This means that the (unsigned) circuits are precisely the 5-subsets of $[m]$.
The (unsigned) cocircuits are precisely the subsets of $[m]$ which do not intersect any circuit in exactly one element, and so are the $(m - 3)$-subsets of $[m]$.
As in Definition~\ref{def:hyp+facet}, the hyperplanes are the complements of the cocircuits, and so are exactly the $3$-subsets of~$[m]$.
This final fact will be particularly useful to bear in mind in the following proof.
We will use Lemma~\ref{lem:chi->cocirc} to show that the $3$-subsets described in Proposition~\ref{prop:facets} are the complements of positive cocircuits, which gives that they are facets by Definition~\ref{def:hyp+facet}.

\begin{proof}[{Proof of Proposition~\ref{prop:facets}}]
We first show that every hyperplane given by a triangle $H$ of $T$ is an upper facet.
Let $H = abc$ be a triangle of $T$.
We must show that every $x \in [m] \setminus H$ has the same sign in this cocircuit.
If $c \prec x \prec a$, then $\chi_{T}(x, a, b, c) = + 1$ due to the arc $ac$ of~$T$.
If $a \prec x \prec b$, then $\chi_{T}(x, a, b, c) = - \chi_{T}(a, x, b, c) = + 1$ due to the arc $ab$ of~$T$.
One can similarly verify for $b \prec x \prec c$ that we have $\chi_{T}(x, a, b, c) = + 1$.
We conclude that $abc$ is an upper facet of $\mathcal{M}_{T}$.
Moreover, using Lemma~\ref{lem:chi->cocirc}, one can see that the signatures of $[m] \setminus H$ are $([m] \setminus H, \emptyset)$ and $(\emptyset, [m] \setminus H)$, and so $H$ is indeed a facet of $\mathcal{M}_{T}$.

Now let $H = abc$ be such that $H - 1$ is a triangle of $T$.
If $x \in [m] \setminus H$ has $c \prec x \prec a$, then $\chi_{T}(x, a, b, c) = - 1$ due to the arc $(a - 1)(c - 1)$ of~$T$.
Proceeding similarly to above, it is clear that $x$ must always have the sign $- 1$ whenever it lies in $[m] \setminus H$.
We conclude that $H$ is a lower facet of $\mathcal{M}_{T}$.

We now show that all facets of $\mathcal{M}_{T}$ are given by $H$ for $H$ either a triangle of $T$ or a triangle of $T + 1$.
Indeed, suppose that $abc$ is an upper facet of $\mathcal{M}_{T}$.
By rotating, we may assume that $c = m$ and $a \neq 1$.
Since $abc$ is an upper facet, we have $\chi(1, a, b, c) = +1$.
Then there must be an arc $xc$ of $T$ with $a \preccurlyeq x \prec b$.
If $x \neq a$, then $\chi(x, a, b, c) = -\chi(a, x, b, c) = -1$ due to the arc~$xc$, which contradicts the fact that $abc$ is a facet.
We conclude that $ac$ is an arc of~$T$.
One can show in a similar way that $ab$ and $bc$ are arcs of $T$ if they are not boundary edges of the $m$-gon, so that $abc$ is an upper facet.
A similar argument shows that if $abc$ is a lower facet, then it is a triangle of $T + 1$.
Hence, all facets of $\mathcal{M}_{T}$ are given by triangles of $T$ or $T + 1$.
\end{proof}

Given a basis $\sigma$ of $\mathcal{M}_{T}$, one can define upper and lower facets of $\sigma$ as the upper and lower facets of the matroid $\mathcal{M}_{T}(\sigma)$.
Explicitly, we have the following.

\begin{definition}\label{def:up_low_facets}
Let $\sigma = \{a, b, c, d\}$ be a basis of $\mathcal{M}_{T}$ with $a \prec b \prec c \prec d$.
Then if $\chi_{T}(a, b, c, d) = +1$, we say that $\{a, b, d\}$ and $\{b, c, d\}$ are \emph{upper facets of $\sigma$ with respect to $T$} and $\{a, b, c\}$ and $\{a, c, d\}$ are \emph{lower facets of $\sigma$ with respect to $T$}.
If $\chi_{T}(a, b, c, d) = -1$, then $\chi_{T}(d, a, b, c) = +1$, so the definition reduces to the former case.
\end{definition}

\begin{lemma}\label{lem:one_upper_one_lower}
Given a triangulation $\Delta$ of $\mathcal{M}_{T}$ with $\sigma, \tau \in \Delta$, if $|\sigma \cap \tau| = 3$, then $\sigma \cap \tau$ is an upper facet of one of $\sigma$ or $\tau$, and a lower facet of the other. 
\end{lemma}
\begin{proof}
Suppose for contradiction that $\sigma \cap \tau$ is a lower facet of both $\sigma$ and $\tau$, with the case where it is an upper facet of both being similar.
By Proposition~\ref{prop:restriction}, it suffices to show the claim when restricting to $\sigma \cup \tau$.
However, $|\sigma \cup \tau| = 5$, whence the restriction $\mathcal{M}_{T}(\sigma \cup \tau)$ is the oriented matroid of a three-dimensional cyclic polytope with five vertices by Proposition~\ref{prop:circuits}.
In this case, the notions of upper and lower facets coincide with the usual notions for simplices in cyclic polytopes; see, for instance, \cite[Section~2]{er}.
For cyclic polytopes, if $\sigma \cap \tau$ is a lower facet of both $\sigma$ and $\tau$ then their geometric realisations $||\sigma||$ and $||\tau||$ intersect in their interiors, since points a small distance above $||\sigma||$ and $||\tau||$ must lie in both $||\sigma||$ and $||\tau||$.
(Here we are using the standard realisation of cyclic polytopes as the convex hull of a finite set of points on the moment curve --- see  for instance \cite[Section~2.1]{njw-hst} --- and $||\sigma||$ denotes the geometric realisation of $\sigma$ as the convex hull of its vertex set on the moment curve.)

Since $||\sigma||$ and $||\tau||$ intersect in their interiors, there is therefore a circuit $C = (C^{+}, C^{-})$ such that $\sigma \supseteq C^{+}$ and $\tau \supseteq C^{-}$.
(In other words, either $\sigma$ has a one-dimensional face which intersects a two-dimensional face of $\tau$ or \textit{vice versa}.)
In fact, since $|\underline{C}| = 5 = |\sigma \cup \tau|$ and $|\sigma| = |\tau| = 4$, we have $|\sigma \setminus \tau| = 1$.
We can then let $\sigma \setminus \tau = \{a\}$, and so $a \in C^{+}$ and $\underline{C} \setminus \{a\} \subseteq \tau$.
Hence, $\sigma$ and $\tau$ overlap on a circuit, meaning that they cannot be elements of the same triangulation, and so we have arrived at a contradiction.
\end{proof}

Noting Lemma~\ref{lem:one_upper_one_lower}, we make the following definition.

\begin{definition}
Given a triangulation $\Delta$ of $\mathcal{M}_{T}$, we define a relation $\triangrel$ on $\Delta$ to be the smallest reflexive transitive relation such that $\sigma \triangrel \tau$ whenever $\sigma \cap \tau$ is an upper facet of $\sigma$ and a lower facet of $\tau$.
Following \cite[Definition~2.13]{rs-baues}, we say that a triangulation $\Delta$ of $\mathcal{M}_{T}$ is \emph{stackable} if the relation $\triangrel$ is a partial order on $\Delta$.
It suffices to show that $\triangrel$ is an anti-symmetric relation.
\end{definition}

In fact, we conjecture the following.

\begin{conjecture}
Every triangulation $\Delta$ of $\mathcal{M}_{T}$ is stackable.
\end{conjecture}

\section{Showing the bijection}\label{sect:bijection}

We now show that stackable triangulations of these oriented matroids are in bijection with equivalence classes of maximal green sequences.
Throughout this section we fix an integer $n \geqslant 1$.

\subsection{Preliminary results}

We first show the following key result, which shows the relation between our oriented matroid $\mathcal{M}_{T}$ and the choice of extriangulated structure on $\cc{n}$ coming from the cluster tilting object $\sft$ corresponding to $T + 1$.
This lemma is the motivation for Definition~\ref{def:mt}.

\begin{lemma}\label{lem:key}
Let $\sft \in \cc{n}$ be a cluster-tilting object, and let $T$ be the triangulation of the $(n + 3)$-gon corresponding to $\Sigma\sft$.
Let $\sfo_{A}, \sfo_{B} \in \cc{n}$ be such that $A = \{a, c\}$, $B = \{b, d\}$ with $a \prec b \prec c \prec d$.
Then $\Ext_{\sft}^{1}(\sfo_{B}, \sfo_{A}) \neq 0$ if and only if $\chi_{T}(a, b, c, d) = +1$; and so $\Ext_{\sft}^{1}(\sfo_{B}, \sfo_{A}) = 0$ if and only if $\chi_{T}(a, b, c, d) = -1$.
\end{lemma}
\begin{proof}
Note that by Section~\ref{sect:back:cca}, we have that $\Ext_{\cc{n}}^{1}(\sfo_{B}, \sfo_{A}) \neq 0$ due to the fact that $a \prec b \prec c \prec d$, since, by definition, $\Ext_{\cc{n}}^{1}(\sfo_{B}, \sfo_{A}) = \Hom_{\cc{n}}(\sfo_{B}, \Sigma\sfo_{A})$.
In the extriangulated structure defined by $\sft$, we have that $\Ext_{\sft}^{1}(\sfo_{B}, \sfo_{A}) \neq 0$ if and only if the non-zero map $\sfo_{B} \to \Sigma\sfo_{A}$ factors through $\Sigma\sft$.
This is then the case precisely if there is a non-zero summand $\sfo_{X}$ of $\Sigma\sft$ such that there is a non-zero composition $\sfo_{B} \to \sfo_{X} \to \Sigma\sfo_{A}$, since the map $\sfo_{B} \to \Sigma\sfo_{A}$ is unique up to scalar, by Section~\ref{sect:back:cca}.
By Lemma~\ref{lem:composition} and the description of $\Sigma$ from Section~\ref{sect:back:cca}, this happens if and only if we have $b \preccurlyeq x \prec c \prec d \preccurlyeq y \prec a$, where $X = \{x, y\}$, and this is the case if and only if $\chi_{T}(a, b, c, d) = +1$.
\end{proof}

Here we used the following lemma, whose proof we will postpone to the appendix.

\begin{lemma}\label{lem:composition}
A pair of non-zero morphisms $\sfo_{P} \to \sfo_{Q}$ and $\sfo_{Q} \to \sfo_{R}$ in $\cc{n}$ compose to give a non-zero morphism $\sfo_{P} \to \sfo_{R}$ if and only if $p_{0} - 1 \preccurlyeq q_{0} - 1 \prec r_{0} \prec p_{1} - 1 \preccurlyeq q_{1} - 1 \prec r_{1}$, where $P = \{p_{0}, p_{1}\}$, $Q = \{q_{0}, q_{1}\}$, and $R = \{r_{0}, r_{1}\}$.
\end{lemma}

\begin{definition}\label{def:mt_ct}
Given a cluster-tilting object $\sft$ in $\cc{n}$, we thus define $\mathcal{M}_{\sft}$ to be the oriented matroid $\mathcal{M}_{T}$, where $T$ is the triangulation of the $(n + 3)$-gon corresponding to $\Sigma\sft$.
We similarly write $\chi_{\sft} = \chi_{T}$.
\end{definition}

\subsection{Maximal green sequence to triangulation}

We explain how to construct a triangulation of $\mathcal{M}_{\sft}$ from a maximal green sequence of~$\sft$.

\begin{construction}
Let $\mathcal{G}$ be a maximal green sequence of $\sft$, given by a sequence of cluster-tilting objects $\sft = \sft_{0}, \sft_{1}, \dots, \sft_{l} = \Sigma\sft$ in $\cc{n}$.
Denote the exchange pair of the mutation from $\sft_{i - 1}$ to $\sft_{i}$ by $(\sfo_{A_{i}}, \sfo_{B_{i}})$.
Since we have $\Ext_{\sft}^{1}(\sfo_{B_{i}}, \sfo_{A_{i}}) \neq 0$, the arcs defined by $A_{i}$ and $B_{i}$ cross in the $(n + 3)$-gon, and so $|A_{i} \cup B_{i}| = 4$.
Hence, we define the set \[\Delta_{\mathcal{G}} := \{A_{i} \cup B_{i} \st (\sfo_{A_{i}}, \sfo_{B_{i}}) \in \exch{\mathcal{G}}\}\]
of $3$-simplices of $\mathcal{M}_{\sft}$.
We claim that this is indeed a triangulation of $\mathcal{M}_{\sft}$.
\end{construction}

First, we note that this triangulation respects equivalence of maximal green sequences.

\begin{lemma}
If $\mathcal{G}$ and $\mathcal{G}'$ are equivalent, then $\Delta_{\mathcal{G}} = \Delta_{\mathcal{G}'}$.
\end{lemma}
\begin{proof}
If $\mathcal{G}$ and $\mathcal{G}'$ are equivalent, then $\exch{\mathcal{G}} = \exch{\mathcal{G}'}$, so $\Delta_{\mathcal{G}} = \Delta_{\mathcal{G}'}$.
\end{proof}

Let $T_{0}, T_{1}, \dots, T_{l}$ be the sequence of triangulations of the $(n + 3)$-gon corresponding to the maximal green sequence $\sft_{0}, \sft_{1}, \dots, \sft_{l}$.
Let $\sigma = A_{i} \cup B_{i}$ be the simplex corresponding to the exchange pair $(\sfo_{A_{i}}, \sfo_{B_{i}})$ from $\sft_{i - 1}$ to $\sft_{i}$.
Then the triangulation $T_{i - 1}$ is obtained from the triangulation $T_{i}$ by replacing the arc $A_{i}$ with the arc $B_{i}$.
Hence $T_{i - 1}$ and $T_{i}$ contain the quadrilateral $A_{i} \cup B_{i}$ and one can equivalently say that $T_{i}$ is obtained from $T_{i - 1}$ by replacing the lower facets of $\sigma$ by the upper facets of $\sigma$.
In this case, we say that $T_{i}$ is the \emph{increasing flip} of $T_{i - 1}$ given by $\sigma$.
This perspective will be useful in the proof of the following proposition.

\begin{proposition}\label{prop:delta_g_triang}
We have that $\Delta_{\mathcal{G}}$ is a triangulation of $\mathcal{M}_{\sft}$.
\end{proposition}

In order to prove this proposition, we will rely on the following lemmas on maximal green sequences from \cite{njw-mg}.

\begin{lemma}[{\cite[Lemma~3.4]{njw-mg}}]\label{lem:njw-mg_3.4}
Let $\sft \in \cc{n}$ be a basic cluster-tilting object and let $\sft_{0}, \sft_{1}, \dots, \sft_{l}$ be a maximal green sequence $\mathcal{G}$ of $\sft$.

Then, given an indecomposable object $\sfo \in \cc{n}$, there exist $j, k \in [l]$ with $j \leqslant k$ such that \[
\{\sft_{i} \st \sfo \text{ is a direct summand of } \sft_{i}\} = \{\sft_{j}, \sft_{j + 1}, \dots, \sft_{k}\},
\]
if this set is non-empty.
\end{lemma}

In other words, an indecomposable object $\sfo$ enters $\mathcal{G}$ at most once, leaves $\mathcal{G}$ at most once, and enters before it leaves if it does both.

\begin{lemma}[{\cite[Corollary~3.2(2)]{njw-mg}}]\label{lem:njw-mg_3.2}
Let $\sft \in \cc{n}$ be a basic cluster-tilting object and let $\sft_{0}, \sft_{1}, \dots, \sft_{l}$ be a maximal green sequence $\mathcal{G}$ of $\sft$.

Let $\sfx, \sfy \in \cc{n}$ be indecomposable objects and suppose that there exist $j, k \in [l]$ with $j \leqslant k$ such that $\sfx$ is a direct summand of $\sft_{j}$ and $\sfy$ is a direct summand of~$\sft_{k}$.
In this case, we say that $\sfx$ \emph{precedes} $\sfy$ in $\mathcal{G}$, and we have that $\Ext_{\sft}^{1}(\sfx, \sfy) = 0$.
\end{lemma}

\begin{proof}[Proof of Proposition~\ref{prop:delta_g_triang}]
We verify the properties of Definition~\ref{def:mat_triang}.
%
We first verify that \eqref{op:mat_triang:facet} holds.
Let $T$ be the triangulation of the $(n + 3)$-gon corresponding to $\Sigma\sft = \sft_{l}$, so that $T + 1$ is the triangulation corresponding to $\sft = \sft_{0}$.
Suppose that $\sigma \in \Delta_{\mathcal{G}}$, so that $\sigma = A_{i} \cup B_{i}$ for $(\sfo_{A_{i}}, \sfo_{B_{i}}) \in \exch{\mathcal{G}}$.
Let $F$ be a facet of $\sigma$ containing $A_{i}$; one can argue similarly for facets of $\sigma$ containing~$B_{i}$.
If $F$ is a triangle of $T + 1$, then it is a facet of $\mathcal{M}_{T}$ by Proposition~\ref{prop:facets}.

Thus, suppose that $F$ is not a triangle of $T + 1$.
Then, if $T_{j}$ is the triangulation corresponding to $\sft_{j}$, let $T_{k}$ be the first triangulation where $F$ appears.
By assumption, $k \neq 0$.
There is then an exchange pair $(\sfo_{A_{k}}, \sfo_{B_{k}})$ and a simplex $\sigma' = A_{k} \cup B_{k} \in \Delta_{\mathcal{G}}$.
Since $F$ is a triangle of $T_{k}$ but not of $T_{k - 1}$, we have that $F$ must also be a facet of $\sigma'$, and must contain~$B_{k}$.
This means that $\sigma \neq \sigma'$, since we cannot have $(A_{i}, B_{i}) = (A_{k}, B_{k})$, as $F$ contains $A_{i}$ and~$B_{k}$.
Hence, $F$ is contained in at least two simplices of $\Delta_{\mathcal{G}}$.

We must now show that $F$ cannot be contained in more than two simplices of~$\Delta_{\mathcal{G}}$.
Suppose that we have another simplex $\sigma'' \in \Delta_{\mathcal{G}}$ with $F$ as a facet of~$\sigma''$.
We have a total order on $\Delta_{\mathcal{G}}$ given by the order of the exchange pairs of~$\mathcal{G}$.
When a simplex $\tau$ corresponds to the exchange pair from $\sft_{h - 1}$ to $\sft_{h}$, the lower facets of $\tau$ leave $T_{h - 1}$ and are replaced by the upper facets of $\tau$.
Since we have at least three simplices in $\Delta_{\mathcal{G}}$ which contain $F$, we have that $F$ enters the sequence of triangulations $T_{0}, T_{1}, \dots, T_{l}$, leaves the sequence, and then enters again.
In particular, this must be true for some $1$-dimensional face $A$ of $F$, which we can choose to be an arc in the $(n + 3)$-gon, since $n + 3 \geqslant 4$.
This then gives us an indecomposable object $\sfo_{A}$ which enters $\mathcal{G}$, leaves, and then enters again.
But this contradicts Lemma~\ref{lem:njw-mg_3.4}.

Finally, we must show \eqref{op:mat_triang:int}: that no two simplices $\sigma, \tau \in \Delta$ overlap on a circuit.
Suppose for contradiction that, in fact, there is a circuit $(C^{+}, C^{-})$ of $\mathcal{M}_{T}$ such that $\sigma \supseteq C^{+}$ and $\tau \supseteq \underline{C} \setminus \{x\}$ for some $x \in C^{+}$.
In particular, $\tau \supseteq C^{-}$.
By rotating, we can assume that $\underline{C} = \{a \prec b \prec c \prec d \prec e\}$ with $C^{\pm} = \{a, c, e\}$ and $C^{\mp} = \{b, d\}$ by Proposition~\ref{prop:circuits}.
Without loss of generality, suppose that $C^{+} = \{a, c, e\}$ and $C^{-} = \{b, d\}$.
Suppose now that $\sigma$ precedes $\tau$ in the order given by $\mathcal{G}$.
We have that $\chi_{\sft}(b, c, d, e) = +1$ since the description of $C$ gives us that $\mathcal{M}_{T}(\underline{C})$ is the oriented matroid of a three-dimensional cyclic polytope with five vertices.
Hence $\Ext_{\sft}^{1}(\sfo_{ce}, \sfo_{bd}) \neq 0$ by Lemma~\ref{lem:key}.
However, since $\sigma$ precedes $\tau$, we have that $\sfo_{ce}$ must precede $\sfo_{bd}$ in $\mathcal{G}$, which contradicts Lemma~\ref{lem:njw-mg_3.2}.
If $\tau$ precedes $\sigma$, then note that $\chi_{T}(a, b, c, d) = +1$, so $\Ext_{\sft}^{1}(\sfo_{bd}, \sfo_{ac}) \neq 0$, which contradicts Lemma~\ref{lem:njw-mg_3.2} in a similar way.
\end{proof}

\begin{proposition}
The triangulation $\Delta_{\mathcal{G}}$ is stackable.
\end{proposition}
\begin{proof}
First note that $\mathcal{G}$ gives a total order $\leqslant_{\mathcal{G}}$ on its exchange pairs according to the order in which they appear.
Since the exchange pairs of $\mathcal{G}$ are in bijection with the simplices of $\Delta_{\mathcal{G}}$, we have that $\leqslant_{\mathcal{G}}$ is also a total order on the simplices of $\Delta$.
We claim that the relation $\triangrel$ on $\Delta$ is contained in the relation $\leqslant_{\mathcal{G}}$.
From this it follows that $\triangrel$ is a partial order, since any reflexive transitive relation contained in a total order is a partial order.

In order to show this, it suffices to show that whenever we have two simplices $\sigma$ and $\tau$ of $\Delta$ such that $\sigma \cap \tau$ is an upper facet of $\sigma$ and a lower facet of $\tau$, we have that $\sigma \leqslant_{\mathcal{G}} \tau$.
Let $A_{i}$ be the intersection of the lower facets of $\sigma$, with $B_{i}$ the intersection of the upper facets; similarly, let $A_{j}$ be the intersection of the lower facets of $\tau$, with $B_{j}$ the intersection of the upper facets.

By Definition~\ref{def:up_low_facets} and Lemma~\ref{lem:key}, we have that $(\sfo_{A_{i}}, \sfo_{B_{i}})$ and $(\sfo_{A_{j}}, \sfo_{B_{j}})$ are the respective exchange pairs corresponding to $\sigma$ and~$\tau$.
Let $\sft_{j - 1}$ be the cluster-tilting object of $\mathcal{G}$ on which the mutation of $\sfo_{A_{j}}$ for $\sfo_{B_{j}}$ is performed.
Then, if $T_{j - 1}$ is the triangulation of the $(n + 3)$-gon corresponding to $\sft_{j - 1}$, we have that $T_{j - 1}$ contains the lower facets of~$\tau$.
Since one of the lower facets of $\tau$ is an upper facet of $\sigma$, and each upper facet of $\sigma$ contains~$B_{i}$, we have that $T_{j - 1}$ contains the arc~$B_{i}$.
This means that $\sfo_{B_{i}}$ is a direct summand of $\sft_{j - 1}$, and so Lemma~\ref{lem:njw-mg_3.4} ensures that $(\sfo_{A_{i}}, \sfo_{B_{i}})$ occurs before~$\sft_{j - 1}$, which gives that $\sigma \leqslant_{\mathcal{G}} \tau$.
\end{proof}

\subsection{Triangulation to maximal green sequence}

We now describe the inverse construction, which starts with a stackable triangulation $\Delta$ of $\mathcal{M}_{\sft}$, and outputs an equivalence class of maximal green sequences $[\mathcal{G}]$ of~$\sft$.
As in the previous section, we have fixed a cluster-tilting object $\sft$ in $\cc{n}$, and have let $T$ be the triangulation of the $(n + 3)$-gon corresponding to $\Sigma\sft$.

\begin{proposition}\label{prop:stack->seq}
Let $\Delta$ be a stackable triangulation of $\mathcal{M}_{T}$ with $\leqslant$ a linear extension of the partial order $\triangrel$ on the simplices of $\Delta$.
Letting the total order $\leqslant$ on the simplices of $\Delta$ be $\sigma_{1} < \sigma_{2} < \dots < \sigma_{l}$, there is then a sequence of triangulations $T_{0}, T_{1}, \dots, T_{l}$ of the $(n + 3)$-gon such that
\begin{enumerate*}
\item $T_{0} = T + 1$;\label{op:stack->seq:t0}
\item $T_{i}$ is the increasing flip of $T_{i - 1}$ given by $\sigma_{i}$;\label{op:stack->seq:ti}
\item $T_{l} = T$.\label{op:stack->seq:tl}
\end{enumerate*}
\end{proposition}
\begin{proof}
We construct the sequence $T_{0}, T_{1}, \dots, T_{l}$ inductively by letting $T_{0} = T$ and letting $T_{i}$ be the increasing flip of $T_{i - 1}$ given by $\sigma_{i}$.
We show that this is well-defined, that is: the lower facets of $\sigma_{i}$ are in fact contained in $T_{i - 1}$.

For the base case, $i = 1$, let $F$ be a lower facet of $\sigma_{1}$.
By Definition~\ref{def:mat_triang}\eqref{op:mat_triang:facet}, we have that $F$ is either a facet of $\mathcal{M}_{T}$, or is contained in precisely two simplices of $\Delta$.
If $F$ is contained in another simplex $\sigma_{j}$ of $\Delta$, then it must be an upper facet of $\sigma_{j}$ by Lemma~\ref{lem:one_upper_one_lower}.
However, then we would have $\sigma_{j} \triangrel \sigma_{1}$, which would imply that $\sigma_{j} < \sigma_{1}$, since $\leqslant$ is a linear extension of $\triangrel$.
Hence, all lower facets of $\sigma_{1}$ are facets of $\mathcal{M}_{T}$.
It is clear from the definition of lower facets that they are in fact lower facets of $\mathcal{M}_{T}$.
Since the lower facets of $\mathcal{M}_{T}$ are given by $T + 1$, the base case holds.

Now, for the inductive step, let $F$ be one of the lower facets of $\sigma_{i}$.
Again, we have that $F$ is either a facet of $\mathcal{M}_{T}$, or is contained in precisely two simplices of $\Delta$.

Suppose first that $F$ is a facet of $\mathcal{M}_{T}$.
Then it must be a lower facet of $\mathcal{M}_{T}$.
Thus, $F$ is a triangle of $T + 1$.
Moreover, $F$ cannot be contained in any other simplices of $\Delta$; in particular, this would contradict Lemma~\ref{lem:one_upper_one_lower}.
Hence, $F$ must still be a triangle of $T_{i - 1}$, since it could not have been removed by any other simplex.

Now suppose that $F$ is contained in precisely one other simplex $\sigma_{j}$ of $\Delta$.
By Lemma~\ref{lem:one_upper_one_lower}, we have that $F$ must be an upper facet of $\sigma_{j}$, and so we have $\sigma_{j} \leqslant \sigma_{i}$.
By the induction hypothesis, we have that the lower facets of $\sigma_{j}$ are contained in $T_{j - 1}$, and so $T_{j}$ contains the upper facets of $\sigma_{j}$, and so contains $F$.
Since $F$ is only contained in $\sigma_{j}$ and $\sigma_{i}$ by Definition~\ref{def:mat_triang}\eqref{op:mat_triang:facet}, it must still be a triangle of $T_{i - 1}$.

We finally must show that \eqref{op:stack->seq:tl} holds.
Indeed, every upper facet of $\mathcal{M}_{T}$ must be an upper facet of some $\sigma_{i}$ by Definition~\ref{def:mat_triang}\eqref{op:mat_triang:facet} and must be an upper facet of exactly one $\sigma_{i}$ by Lemma~\ref{lem:one_upper_one_lower}.
Hence, for every upper facet $F$ of $\mathcal{M}_{T}$, there must be some triangulation $T_{i - 1}$ which contains the lower facets of the simplex $\sigma_{i}$, so that $T_{i}$ contains the upper facets of the simplex $\sigma_{i}$, and so contains $F$.
After $\sigma_{i}$, there is no further $\sigma_{j}$ with $F$ as a facet, and so $T_{l}$ contains $F$.
Therefore $T_{l}$ contains all upper facets of $\mathcal{M}_{T}$, and so $T_{l} = T$, as desired.
\end{proof}

\begin{construction}
Let $\Delta$ be a stackable triangulation of $\mathcal{M}_{T}$.
As before, since $\triangrel$ is a partial order on $\Delta$, we may extend it to a total order $\leqslant$.
Let $\sigma_{1}, \sigma_{2}, \dots, \sigma_{l}$ be the simplices of $\Delta$, ordered according to $\leqslant$.

We then take the sequence $T_{0}$, $T_{1}$, \dots, $T_{l}$ of triangulations from Proposition~\ref{prop:stack->seq}.
By Section~\ref{sect:back:cca} and Lemma~\ref{lem:key}, this corresponds to a sequence of cluster-tilting objects $\sft_{0}$, $\sft_{1}$, \dots, $\sft_{l}$ of $\cc{n}$ such that $\sft_{0} = \sft$, $\sft_{i}$ is a green mutation of $\sft_{i - 1}$, and $\sft_{l} = \Sigma\sft$.
This therefore gives a maximal green sequence $\mathcal{G}^{\leqslant}_{\Delta}$ of $\sft$, and hence an equivalence class~$[\mathcal{G}^{\leqslant}_{\Delta}]$.
\end{construction}

\begin{lemma}
The equivalence class of maximal green sequences $[\mathcal{G}^{\leqslant}_{\Delta}]$ is independent of the linear extension $\leqslant$ chosen of the partial order $\triangrel$.
\end{lemma}
\begin{proof}
If $\leqslant'$ is a different linear extension, then we have $\exch{\mathcal{G}^{\leqslant}_{\Delta}} = \exch{\mathcal{G}^{\leqslant'}_{\Delta}}$, since both correspond to the set of exchange pairs $(\sfo_{A_{i}}, \sfo_{B_{i}})$ where $\sigma_{i} = A_{i} \cup B_{i}$ is a simplex of $\Delta$, with $A_{i}$ the intersection of its lower facets and $B_{i}$ the intersection of its upper facets.
In turn, this is because, for either $\leqslant$ or $\leqslant'$, if $T_{i}$ is the increasing flip of $T_{i - 1}$ given by~$\sigma_{i}$, then $T_{i}$ is the result of replacing the arc $A_{i}$ of $T_{i - 1}$ with the arc~$B_{i}$.
By the relation between triangulations of the $(n + 3)$-gon and cluster-tilting objects from Section~\ref{sect:back:cca}, we then have that the corresponding cluster-tilting objects $\sft_{i}$ and $\sft_{i - 1}$ are related by the mutation with exchange pair $(\sfo_{A_{i}}, \sfo_{B_{i}})$.
\end{proof}

Since our constructions in this subsection and the previous are inverse to each other, we obtain the following theorem.

\begin{theorem}\label{thm:main}
Let $\sft$ be a basic cluster-tilting object in $\cc{n}$.
There is then a bijection between stackable triangulations of $\mathcal{M}_{\sft}$ and equivalence classes of maximal green sequences of~$\sft$.
\end{theorem}

\appendix

\section{Module, derived, and cluster categories of type~$\mathbb{A}$}

In this appendix, we prove Lemma~\ref{lem:composition} using combinatorial models for the representation theory of the path algebra of linearly oriented type~$\mathbb{A}$.
Our aim is partly expository, for readers less familiar with quiver representations.
We denote the quiver \[
1 \leftarrow 2 \leftarrow \dots \leftarrow n
\]
by $\mathbb{A}_{n}$ and its path algebra over a field $K$ by $K\mathbb{A}_{n}$.

\subsection{The module category}

For background on quiver representations we refer to \cite{schiff} or \cite{ass}.
It is well-known that finite-dimensional right $KQ$-modules are equivalent to finite-dimensional representations of $Q$ over~$K$ \cite[Chapter~III, Theorem~1.6]{ass}.
It is also well-known that the indecomposable representations of $\mathbb{A}_{n}$ are given by \[
0 \to \dots \to \overset{i}{K} \xleftarrow{\id_{K}} \overset{i + 1}{K} \xleftarrow{\id_{K}} \dots \xleftarrow{\id_{K}} \overset{j}{K} \to 0 \to \dots \to 0,
\]
where $1 \leqslant i \leqslant j \leqslant n$.
This follows for instance from Gabriel's Theorem \cite{gabriel}.
We denote this indecomposable representation by $\sfm_{i, j + 2}$.
Hence, the indecomposable representations of $\mathbb{A}_{n}$ are $\{\sfm_{ac} \st a, c \in [n + 2] \text{ with } a + 2 \leqslant c\}$.

Suppose we have a homomorphism of indecomposable representations representations $f \colon \sfm_{ac} \to \sfm_{bd}$ given by
\begin{equation}\label{eq:rep_hom}
\begin{tikzcd}
V_{1} \ar[d,"{f_{1}}"] & V_{2} \ar[l,"{g_{1}}"] \ar[d,"{f_{2}}"] & \dots \ar[l,"{g_{2}}"] & V_{n} \ar[d,"{f_{n}}"] \ar[l,"{g_{n - 1}}"] \\
W_{1} & W_{2} \ar[l,"{h_{1}}"] & \dots \ar[l,"{h_{2}}"] & W_{n}, \ar[l,"{h_{n - 1}}"] 
\end{tikzcd}
\end{equation}
where $\sfm_{ac}$ is $V_{1} \leftarrow \dots \leftarrow V_{n}$ and $\sfm_{bd}$ is $W_{1} \leftarrow \dots \leftarrow W_{n}$.
If $a - 1 < b < c - 1 < d$, there is a non-zero homomorphism $f_{ac, bd} \colon \sfm_{ac} \to \sfm_{bd}$ which factorises as \[
\sfm_{ac} \twoheadrightarrow \sfm_{bc} \hookrightarrow \sfm_{bd}.
\]
This homomorphism is in fact unique up to scalar, since the choice of scalar for one of the non-zero components $f_{i}$ for $i \in [b, c]$ determines all of the others, since the squares in the diagram commute.
By inspecting the diagram \eqref{eq:rep_hom}, note further that $\Hom_{K\mathbb{A}_{n}}(\sfm_{ac}, \sfm_{bd}) = 0$ if any of the following hold: $b \leqslant a - 1$, $c \leqslant b + 1$, $d \leqslant c - 1$.
We have therefore shown the following lemma.

\begin{lemma}[{\cite[Theorem~3.6(3)]{ot}}]\label{lem:hom_existence}
We have that $\Hom_{K\mathbb{A}_{n}}(\sfm_{ac}, \sfm_{bd}) \not\cong 0$ if and only if $a - 1 < b < c - 1 < d$, and in this case $\Hom_{K\mathbb{A}_{n}}(\sfm_{ac}, \sfm_{bd}) \cong K$.
\end{lemma}

We can go further and give a criterion for when two non-zero morphisms compose to give a non-zero morphisms.
(Experts will note that this is just another way of rewriting the well-known relations for the Auslander algebra.)

\begin{lemma}\label{lem:module_comp}
A pair of non-zero morphisms $\sfm_{p_{0}p_{1}} \to \sfm_{q_{0}q_{1}}$ and $\sfm_{q_{0}q_{1}} \to \sfm_{r_{0}r_{1}}$ compose to give a non-zero morphism $\sfm_{p_{0}p_{1}} \to \sfm_{r_{0}r_{1}}$ if and only if $p_{0} - 1 \leqslant q_{0} - 1 < r_{0} < p_{1} - 1 \leqslant q_{1} - 1 < r_{1}$.
\end{lemma}
\begin{proof}
We have that $p_{0} - 1 \leqslant q_{0} - 1 < r_{0} < p_{1} - 1 \leqslant q_{1} - 1 < r_{1}$ if and only if the following three sets of inequalities all hold.
\begin{align*}
&p_{0} - 1 < q_{0} < p_{1} - 1 < q_{1}, \\
&q_{0} - 1 < r_{0} < q_{1} - 1 < r_{1}, \\
&p_{0} - 1 < r_{0} < p_{1} - 1 < r_{1}. \\
\end{align*}
By Lemma~\ref{lem:hom_existence}, we have that these three sets of inequalities are equivalent to the non-vanishing of the Hom-spaces
\begin{align*}
\Hom_{K\mathbb{A}_{n}}(\sfm_{p_{0}p_{1}}, \sfm_{q_{0}q_{1}}) &\neq 0, \\
\Hom_{K\mathbb{A}_{n}}(\sfm_{q_{0}q_{1}}, \sfm_{r_{0}r_{1}}) &\neq 0, \\
\Hom_{K\mathbb{A}_{n}}(\sfm_{p_{0}p_{1}}, \sfm_{r_{0}r_{1}}) &\neq 0.
\end{align*}

Hence, the lemma is equivalent to the statement that two morphisms $\sfm_{p_{0}p_{1}} \to \sfm_{q_{0}q_{1}}$ and $\sfm_{q_{0}q_{1}} \to \sfm_{r_{0}r_{1}}$ to give a non-zero morphism $\sfm_{p_{0}p_{1}} \to \sfm_{r_{0}r_{1}}$ if and only if there exists a non-zero morphism $\sfm_{p_{0}p_{1}} \to \sfm_{r_{0}r_{1}}$.
The forwards direction is immediate.
For the backwards direction, suppose that we have two non-zero morphisms $\sfm_{p_{0}p_{1}} \to \sfm_{q_{0}q_{1}}$ and $\sfm_{q_{0}q_{1}} \to \sfm_{r_{0}r_{1}}$.
By Lemma~\ref{lem:hom_existence}, these morphisms are unique up to scalar, so we may assume without loss of generality that these morphisms are $f_{p_{0}p_{1},q_{0}q_{1}}$ and $f_{q_{0}q_{1},r_{0}r_{1}}$.
Then suppose that we have that $\Hom_{K\mathbb{A}_{n}}(\sfm_{p_{0}p_{1}}, \sfm_{r_{0}r_{1}}) \neq 0$, so that in particular we have the non-zero morphism $f_{p_{0}p_{1}, r_{0}r_{1}}$.

It suffices to show that $f_{q_{0}q_{1}, r_{0}r_{1}} \circ f_{p_{0}p_{1}, q_{0}q_{1}} = f_{p_{0}p_{1}, r_{0}r_{1}}$.
This follows from the fact that the diagram \[
\begin{tikzcd}
\sfm_{p_{0}p_{1}} \ar[rd,twoheadrightarrow] && \sfm_{q_{0}q_{1}} \ar[rd,twoheadrightarrow] && \sfm_{r_{0}r_{1}} \\
& \sfm_{q_{0}p_{1}} \ar[rd,twoheadrightarrow] \ar[ru,hookrightarrow] && \sfm_{r_{0}q_{1}} \ar[ru,hookrightarrow] & \\
&& \sfm_{r_{0}p_{1}} \ar[ru,hookrightarrow] &&
\end{tikzcd}
\]
commutes.
Indeed, $f_{p_{0}p_{1}, q_{0}q_{1}}$ is the composition $\sfm_{p_{0}p_{1}} \twoheadrightarrow \sfm_{q_{0}p_{1}} \hookrightarrow \sfm_{q_{0}q_{1}}$; $f_{q_{0}q_{1}, r_{0}r_{1}}$ is the composition $\sfm_{q_{0}q_{1}} \twoheadrightarrow \sfm_{r_{0}q_{1}} \hookrightarrow \sfm_{r_{0}r_{1}}$; and $f_{p_{0}p_{1},r_{0}r_{1}}$ is the composition $\sfm_{p_{0}p_{1}} \twoheadrightarrow \sfm_{q_{0}p_{1}} \twoheadrightarrow \sfm_{r_{0}p_{1}} \hookrightarrow \sfm_{r_{0}q_{1}} \hookrightarrow \sfm_{r_{0}r_{1}}$.
Commutativity of the middle square follows easily from considering the explicit quiver representations: the morphism given by either path around the square is precisely the map which is zero on the vector spaces at vertices in $[q_{0}, r_{0} - 1]$ and the identity on the vector spaces at vertices in $[r_{0}, p_{1} - 2]$.
\end{proof}

\subsection{The derived category}

We now describe the bounded derived category $\mathscr{D}_{n} := \mathscr{D}^{b}(K\mathbb{A}_{n})$ of right $K\mathbb{A}_{n}$-modules.
For $a, c \in [n + 2]$ with $a + 2 \leqslant c$, we define $\sfu_{ac}$ to be the complex given by $\sfm_{ac}$ concentrated in degree zero.
We then define
\begin{align}
\sfu_{ac}[1] &= \sfu_{c - 1, a + n + 2}, \\
\sfu_{ac}[-1] &= \sfu_{c - (n + 2), a + 1}.\label{eq:[-1]_desc}
\end{align}
Hence, every complex which is given by an indecomposable $K\mathbb{A}_{n}$-module concentrated in one degree is labelled as $\sfu_{ac}$ for some $a,c$.
By a result of Happel \cite[Lemma~5.2]{happel}, indecomposable objects in $\mathscr{D}_{n}$ are precisely given by complexes of indecomposable $K\mathbb{A}_{n}$-modules concentrated in one degree.
Hence, the labelling of the indecomposable objects $\sfu_{ac}$ gives a bijection between the indecomposable objects of $\mathscr{D}_{n}$ and the elements of $\{(a, c) \in \mathbb{Z}^{2} \st a + 2 \leqslant c \leqslant a + n + 1\}$.

This labelling is a special case of the labelling in \cite[Section~6]{ot}, as follows from combining \cite[Lemma~6.6]{ot} and \cite[Lemma~6.7]{ot}.
We then have the following analogue of Lemma~\ref{lem:hom_existence} for $\mathscr{D}_{n}$.

\begin{lemma}[{\cite[Proof of Proposition~6.1]{ot}}]\label{lem:der_hom}
We have $\Hom_{\mathscr{D}_{n}}(\sfu_{ac}, \sfu_{bd}) \not\cong 0$ if and only if $a - 1 < b < c - 1 < d < a + n + 2$, and in this case $\Hom_{\mathscr{D}_{n}}(\sfu_{ac}, \sfu_{bd}) \cong K$.
\end{lemma}

Recalling the Nakayama functor $\nu$ from Section~\ref{sect:back:cca}, we denote by $\tau$ the functor $\nu[-1]$, which is an auto-equivalence of~$\mathscr{D}_{n}$ known as the \emph{Auslander--Reiten translate}.
It admits the following combinatorial description.

\begin{lemma}[{\cite[Lemma~6.6]{ot}}]\label{lem:tau}
We have that $\tau \sfu_{ac} = \sfu_{a - 1, c - 1}$.
\end{lemma}

We can then use this to prove the following analogue of Lemma~\ref{lem:module_comp}.

\begin{lemma}\label{lem:der_comp}
A pair of non-zero morphisms $\sfu_{p_{0}p_{1}} \to \sfu_{q_{0}q_{1}}$ and $\sfu_{q_{0}q_{1}} \to \sfu_{r_{0}r_{1}}$ compose to give a non-zero morphism $\sfu_{p_{0}p_{1}} \to \sfu_{r_{0}r_{1}}$ if and only if $p_{0} - 1 \leqslant q_{0} - 1 < r_{0} < p_{1} - 1 \leqslant q_{1} - 1 < r_{1} < p_{0} + n + 2$.
\end{lemma}
\begin{proof}
Suppose that we have $\sfu_{p_{0}p_{1}} \to \sfu_{q_{0}q_{1}}$ and $\sfu_{q_{0}q_{1}} \to \sfu_{r_{0}r_{1}}$ composing to give a non-zero morphism $\sfu_{p_{0}p_{1}} \to \sfu_{r_{0}r_{1}}$.
Then, by applying the auto-equivalence $\tau^{p_{0} - 1}$, we can assume that $p_{0} = 1$, so that $\sfu_{p_{0}p_{1}}$ is concentrated in degree zero.
Then the existence of non-zero morphisms $\sfu_{p_{0}p_{1}} \to \sfu_{q_{0}q_{1}}$ gives by Lemma~\ref{lem:der_hom} that $p_{0} \leqslant q_{0}$ and $q_{1} \leqslant p_{0} + n + 1 = n + 2$, so that $\sfu_{q_{0}q_{1}}$ is concentrated in degree zero too.
We likewise deduce that $\sfu_{r_{0}r_{1}}$ is concentrated in degree zero.
We then use Lemma~\ref{lem:module_comp} to deduce that $p_{0} - 1 \leqslant q_{0} - 1 < r_{0} < p_{1} - 1 \leqslant q_{1} - 1 < r_{1}$.
The fact that $r_{1} < n + 3 = p_{0} + n + 2$ follows from the fact that $\sfu_{r_{0}r_{1}}$ is concentrated in degree zero.

Conversely, suppose that $p_{0} - 1 \leqslant q_{0} - 1 < r_{0} < p_{1} - 1 \leqslant q_{1} - 1 < r_{1} < p_{0} + n + 2$.
Then we may again apply $\tau^{p_{0} - 1}$, whence these inequalities give us that all of $\sfu_{p_{0}p_{1}}$, $\sfu_{q_{0}q_{1}}$, and $\sfu_{r_{0}r_{1}}$ are concentrated in degree zero.
We can then apply Lemma~\ref{lem:module_comp} to deduce that we have a pair of non-zero morphisms $\sfu_{p_{0}p_{1}} \to \sfu_{q_{0}q_{1}}$ and $\sfu_{q_{0}q_{1}} \to \sfu_{r_{0}r_{1}}$ composing to give a non-zero morphism $\sfu_{p_{0}p_{1}} \to \sfu_{r_{0}r_{1}}$.
\end{proof}

\subsection{The cluster category}

We now deduce similar results for the cluster category $\cc{n}$.
The following description of $\nu[-2]$ can be obtained from combining \eqref{eq:[-1]_desc} with Lemma~\ref{lem:tau}.

\begin{lemma}[{\cite[Lemma~6.7]{ot}}]\label{lem:cluster_rep}
We have that $\nu \sfu_{ac} [-2] = \sfu_{c - (n + 3), a}$.
\end{lemma}

Let $\psi \colon \mathscr{D}_{n} \to \cc{n}$ be the natural projection functor sending an object to its orbit.
Lemma~\ref{lem:cluster_rep} then allows us to label the indecomposable objects of the $\cc{n}$ as $\sfo_{ac}$ such that
\begin{equation}\label{eq:modulo_cc}
\psi(\sfu_{a'c'}) = \sfo_{ac} \text{ if and only if } \{a', c'\} \equiv \{a, c\} \mod (n + 3).
\end{equation}
The condition for $\Hom_{\cc{n}}(\sfo_{ac}, \sfo_{bd}) \not\cong 0$ from Section~\ref{sect:back:cca} ($a - 1 \prec b \prec c - 1 \prec d$) then follows from \eqref{eq:modulo_cc} with Lemma~\ref{lem:der_hom}.
As always, we have that the non-zero morphism is unique up to scalar if it exists.

We can now prove Lemma~\ref{lem:composition}, which is the analogue of Lemma~\ref{lem:module_comp} and Lemma~\ref{lem:der_comp} for the cluster category.

\begin{proof}[Proof of Lemma~\ref{lem:composition}]
Suppose that $\sfo_{P} \to \sfo_{Q}$ and $\sfo_{Q} \to \sfo_{R}$ in $\cc{n}$ compose to give a non-zero morphism $\sfo_{P} \to \sfo_{R}$.
Then, since $\cc{n}$ is an orbit category of $\mathscr{D}_{n}$, we can lift these morphisms to morphisms $\sfu_{P'} \to \sfu_{Q'}$ and $\sfu_{Q'} \to \sfu_{R'}$ in $\mathscr{D}_{n}$ which compose to give a non-zero map, where $\psi(\sfu_{P'}) = \sfo_{P}$, $\psi(\sfu_{Q'}) = \sfo_{Q}$, and $\psi(\sfu_{R'}) = \sfo_{R}$.
We have non-zero morphisms $\sfu_{P'} \to \sfu_{Q'}$ and $\sfu_{Q'} \to \sfu_{R'}$ in $\mathscr{D}_{n}$ composing to give a non-zero morphism $\sfu_{P'} \to \sfu_{R'}$ if and only if 
\begin{equation}\label{eq:composition:inequality}
p'_{0} - 1 \leqslant q'_{0} - 1 < r'_{0} < p'_{1} - 1 \leqslant q'_{1} - 1 < r'_{1} < p'_{0} + n + 2
\end{equation}
by Lemma~\ref{lem:der_comp}, where $P' = p'_{0}p'_{1}$, $Q' = q'_{0}q'_{1}$, and $R' = r'_{0}r'_{1}$.
Applying \eqref{eq:modulo_cc} gives us the desired cyclic ordering $p_{0} - 1 \preccurlyeq q_{0} - 1 \prec r_{0} \prec p_{1} - 1 \preccurlyeq q_{1} - 1 \prec r_{1}$.

Conversely, if such a cyclic ordering holds, one can lift $\sfo_{P}$, $\sfo_{Q}$, and $\sfo_{R}$ to objects $\sfu_{P'}$, $\sfu_{Q'}$, and $\sfu_{R'}$ of $\mathscr{D}_{n}$ such that \eqref{eq:composition:inequality} holds.
This gives that the morphisms in $\mathscr{D}_{n}$ compose to give a non-zero morphism by Lemma~\ref{lem:der_comp}, and so the morphisms in $\cc{n}$ likewise compose to give a non-zero morphism.
\end{proof}

\bibliographystyle{alpha}
\bibliography{/home/njw/Documents/Offline_versions/Jabref/biblio.bib}

\newcommand{\etalchar}[1]{$^{#1}$}
\begin{thebibliography}{BLVS{\etalchar{+}}99}

\bibitem[AE24]{ae_tcg}
Dario Antolini and Nick Early.
\newblock The chirotropical {G}rassmannian, 2024.
\newblock arXiv:2411.07293.

\bibitem[And25]{anderson}
Laura Anderson.
\newblock {\em Oriented matroids}, volume 216 of {\em Cambridge Studies in
  Advanced Mathematics}.
\newblock Cambridge University Press, Cambridge, 2025.

\bibitem[ASS06]{ass}
Ibrahim Assem, Daniel Simson, and Andrzej Skowro\'{n}ski.
\newblock {\em Elements of the representation theory of associative algebras.
  {V}olume~1: Techniques of representation theory}, volume~65 of {\em London
  Mathematical Society Student Texts}.
\newblock Cambridge University Press, Cambridge, 2006.

\bibitem[BLVS{\etalchar{+}}99]{blswz}
Anders Bj\"{o}rner, Michel Las~Vergnas, Bernd Sturmfels, Neil White, and
  G\"{u}nter~M. Ziegler.
\newblock {\em Oriented matroids}, volume~46 of {\em Encyclopedia of
  Mathematics and its Applications}.
\newblock Cambridge University Press, Cambridge, second edition, 1999.

\bibitem[BMR{\etalchar{+}}06]{bmrrt}
Aslak~Bakke Buan, Bethany Marsh, Markus Reineke, Idun Reiten, and Gordana
  Todorov.
\newblock Tilting theory and cluster combinatorics.
\newblock {\em Adv. Math.}, 204(2):572--618, 2006.

\bibitem[Bre73]{breen}
Marilyn Breen.
\newblock Primitive {R}adon partitions for cyclic polytopes.
\newblock {\em Israel J. Math.}, 15:156--157, 1973.

\bibitem[CCS06]{ccs}
Philippe Caldero, Fr\'ed\'eric Chapoton, and Ralf Schiffler.
\newblock Quivers with relations arising from clusters ({$A_n$} case).
\newblock {\em Trans. Amer. Math. Soc.}, 358(3):1347--1364, 2006.

\bibitem[CCV11]{ccv}
Sergio Cecotti, Clay Cordova, and Cumrun Vafa.
\newblock Braids, walls, and mirrors, 2011.
\newblock arXiv:1110.2115.

\bibitem[CEZ24]{cez}
Freddy Cachazo, Nick Early, and Yong Zhang.
\newblock Color-dressed generalized biadjoint scalar amplitudes: local
  planarity.
\newblock {\em SIGMA Symmetry Integrability Geom. Methods Appl.}, 20:Paper No.
  016, 44, 2024.

\bibitem[ER96]{er}
Paul~H. Edelman and Victor Reiner.
\newblock The higher {S}tasheff--{T}amari posets.
\newblock {\em Mathematika}, 43(1):127--154, 1996.

\bibitem[FGP{\etalchar{+}}24]{fgppp}
Xin Fang, Mikhail Gorsky, Yann Palu, Pierre-Guy Plamondon, and Matthew
  Pressland.
\newblock Extriangulated ideal quotients, with applications to cluster theory
  and gentle algebras, 2024.
\newblock arXiv:2308.05524.

\bibitem[FST08]{fst}
Sergey Fomin, Michael Shapiro, and Dylan Thurston.
\newblock Cluster algebras and triangulated surfaces. {I}. {C}luster complexes.
\newblock {\em Acta Math.}, 201(1):83--146, 2008.

\bibitem[FZ02]{fz1}
Sergey Fomin and Andrei Zelevinsky.
\newblock Cluster algebras. {I}. {F}oundations.
\newblock {\em J. Amer. Math. Soc.}, 15(2):497--529 (electronic), 2002.

\bibitem[Gab72]{gabriel}
Peter Gabriel.
\newblock Unzerlegbare {D}arstellungen. {I}.
\newblock {\em Manuscripta Math.}, 6:71--103; correction, ibid. 6 (1972), 309,
  1972.

\bibitem[GW23]{njw-mg}
Mikhail Gorsky and Nicholas~J. Williams.
\newblock A structural view of maximal green sequences, 2023.
\newblock arXiv:2301.08681.

\bibitem[Hap88]{happel}
Dieter Happel.
\newblock {\em Triangulated categories in the representation theory of
  finite-dimensional algebras}, volume 119 of {\em London Mathematical Society
  Lecture Note Series}.
\newblock Cambridge University Press, Cambridge, 1988.

\bibitem[HLN21]{hln}
Martin Herschend, Yu~Liu, and Hiroyuki Nakaoka.
\newblock $n$-exangulated categories ({I}): Definitions and fundamental
  properties.
\newblock {\em J. Algebra}, 570:531--586, 2021.

\bibitem[Kel11]{kel-green}
Bernhard Keller.
\newblock On cluster theory and quantum dilogarithm identities.
\newblock In {\em Representations of algebras and related topics}, EMS Ser.
  Congr. Rep., pages 85--116. Eur. Math. Soc., Z\"{u}rich, 2011.

\bibitem[KS08]{ks_stability}
Maxim Kontsevich and Yan Soibelman.
\newblock Stability structures, motivic {D}onald\-son--{T}homas invariants and
  cluster transformations, 2008.
\newblock arXiv:0811.2435.

\bibitem[NP19]{np}
Hiroyuki Nakaoka and Yann Palu.
\newblock Extriangulated categories, {H}ovey twin cotorsion pairs and model
  structures.
\newblock {\em Cah. Topol. G\'{e}om. Diff\'{e}r. Cat\'{e}g.}, 60(2):117--193,
  2019.

\bibitem[OT12]{ot}
Steffen Oppermann and Hugh Thomas.
\newblock Higher-dimensional cluster combinatorics and representation theory.
\newblock {\em J. Eur. Math. Soc. (JEMS)}, 14(6):1679--1737, 2012.

\bibitem[PPPP23]{pppp}
Arnau Padrol, Yann Palu, Vincent Pilaud, and Pierre-Guy Plamondon.
\newblock Associahedra for finite-type cluster algebras and minimal relations
  between {$g$}-vectors.
\newblock {\em Proc. Lond. Math. Soc. (3)}, 127(3):513--588, 2023.

\bibitem[Ram97]{rambau}
J\"{o}rg Rambau.
\newblock Triangulations of cyclic polytopes and higher {B}ruhat orders.
\newblock {\em Mathematika}, 44(1):162--194, 1997.

\bibitem[Rei10]{reineke_poisson}
Markus Reineke.
\newblock Poisson automorphisms and quiver moduli.
\newblock {\em J. Inst. Math. Jussieu}, 9(3):653--667, 2010.

\bibitem[RS00]{rs-baues}
J\"{o}rg Rambau and Francisco Santos.
\newblock The generalized {B}aues problem for cyclic polytopes. {I}.
\newblock {\em European J. Combin.}, 21(1):65--83, 2000.

\bibitem[San02]{santos_tom}
Francisco Santos.
\newblock Triangulations of oriented matroids.
\newblock {\em Mem. Amer. Math. Soc.}, 156(741):viii+80, 2002.

\bibitem[Sch14]{schiff}
Ralf Schiffler.
\newblock {\em Quiver representations}.
\newblock CMS Books in Mathematics/Ouvrages de Math{\'e}matiques de la SMC.
  Springer, Cham, 2014.

\bibitem[Wil22]{njw-hst}
Nicholas~J. Williams.
\newblock New interpretations of the higher {S}tasheff--{T}a\-m\-ari orders.
\newblock {\em Adv. Math.}, 407:Paper No. 108552, 49, 2022.

\end{thebibliography}

\end{document}